\documentclass[reqno,fleqn]{amsart}
\usepackage{acronym}
\usepackage{amsthm}
\usepackage[usenames,dvipsnames]{pstricks}
\usepackage{epsfig}
\usepackage{pst-grad} 
\usepackage{pst-plot} 
\usepackage[english]{babel}
\usepackage{graphicx,subfigure}
\usepackage{tabularx}
\usepackage{amsmath,amssymb}
\usepackage{algorithm,algpseudocode} 
\usepackage{fullpage}

\newcommand{\ve}{\varepsilon}
\newcommand{\wt}{\widetilde}

\newcommand{\supp}{\mathrm{supp}}

\makeatletter
\newcommand{\sech}{\mathop{\operator@font sech}}
\newcommand{\sign}{\mathop{\operator@font sign}}
\makeatother

\newtheorem{lemma}{Lemma}[section]
\newtheorem{theorem}{Theorem}[section]
\newtheorem{proposition}{Proposition}[section]

\acrodef{Serre}{Serre}
\acrodef{SGN}{Serre-Green-Naghdi}
\acrodef{cB}{`classical' Boussinesq}
\acrodef{FEM}{Finite Element Method}
\acrodef{IBVP}{initial-boundary value problem}
\acrodef{RK}{Runge-Kutta}
\acrodef{ODEs}{ordinary differential equations}
\acrodef{BT}{Boussinesq type}

\begin{document}

\title[Error estimates for Galerkin approximations of the Serre equations]{Error estimates for Galerkin approximations of the Serre equations}

\author{Dimitrios Antonopoulos}
\address{Mathematics Department, National and Kapodistrian University of Athens, 
15784 Zographou, Greece and Institute of Applied and Computational Mathematics, FORTH, 70013 Heraklion, Greece}
\email{antonod@math.uoa.gr}

\author{Vassilios Dougalis}
\address{Mathematics Department, National and Kapodistrian University of Athens, 
15784 Zographou, Greece and Institute of Applied and Computational Mathematics, FORTH, 70013 Heraklion, Greece}
\email{doug@math.uoa.gr}

\author{Dimitrios Mitsotakis}
\address{Victoria University of Wellington, School of Mathematics and Statistics, Wellington 6140, New Zealand}
\email{dimitrios.mitsotakis@vuw.ac.nz}
\urladdr{http://dmitsot.googlepages.com/}



\begin{abstract}
We consider the Serre system of equations which is a nonlinear dispersive system that models
two-way propagation of long waves of not necessarily small amplitude on the surface of an 
ideal fluid in a channel. We discretize in space the periodic initial-value problem for the 
system using the standard Galerkin finite element method with smooth splines on a uniform
mesh and prove an optimal-order $L^{2}$-error estimate for the resulting semidiscrete 
approximation. Using the fourth-order accurate, explicit, `classical' Runge-Kutta scheme
for time stepping we construct a highly accurate fully discrete scheme in order to
approximate solutions of the system, in particular solitary-wave solutions, and study
numerically phenomena such as the resolution of general initial profiles into sequences of
solitary waves, and overtaking collisions of pairs of solitary waves propagating in the same
direction with different speeds. 
\end{abstract}

\maketitle


\section{Introduction}
In this paper we will analyze standard Galerkin-finite element approximations to the periodic initial-value problem for the
system of {\em Serre equations}. The system consists of two pde's, approximates the two-dimensional 
Euler equations of water-wave theory, and models two-way propagation  of long waves on the surface of
an ideal fluid in a uniform horizontal channel of finite depth $h_{0}$. Specifically, if $\varepsilon=a/h_{0}$, where
$a$ is a typical wave amplitude, and $\sigma=h_{0}/\lambda$, where $\lambda$ is a typical wavelength, the
system is valid when $\sigma\ll 1$ and is written in nondimensional, scaled variables in the form:
\begin{eqnarray}
& & \zeta_{t} + (\eta u)_{x} =0, \label{eq11}\\
& & u_{t} + \zeta_{x} + \ve uu_{x} - \frac{\sigma^{2}}{3\eta} \bigl[ \eta^{3}(u_{xt} + \ve uu_{xx} -\ve u_{x}^{2})\bigr]_{x}=0. \label{eq12}
\end{eqnarray}
Here $x$ and $t$ are proportional to position along the channel and time, respectively, 
$\ve\zeta$, where $\zeta=\zeta(x,t)$, is the
elevation of the free surface above a level of rest at height $y=0$ of the vertical axis, $\eta=1+\ve \zeta$, assumed
to be positive, is the water depth (as the horizontal bottom in these variables is located at $y=-1$), and $u=u(x,t)$ is
the vertically averaged horizontal velocity of the fluid. (For $\ve=O(1)$ the left-hand side of (\ref{eq12}) is an $O(\sigma^{4})$
asymptotic approximation derived from the equation of conservation of momentum in the $x$ direction of the 2D-Euler
equations; (\ref{eq11}) is exact.)\par
The system (\ref{eq11})-(\ref{eq12}) was first derived by Serre, \cite{S}, and subsequently rederived by Su and Gardner, 
\cite{SG}, by Green {\em et al.}, and Green and Naghdi, \cite{GN1}, \cite{GN2}, (who extended it to the case of two spatial variables and
variable bottom), and others. It is also known as {\em Green -- Naghdi} or {\em  fully nonlinear Boussinesq system}.
For its formal derivation from the Euler equations and the derivation of related systems, cf. \cite{LB};  regarding its rigorous
justification as an approximation of the Euler equations we refer the reader to the recent monograph by Lannes, \cite{L},
and its references. \par
In case one considers long waves of {\em small amplitude}, specifically in the {\em Boussinesq} regime
$\ve=O(\sigma^{2})$, $\sigma \ll 1$, it is straightforward to see that the Serre system becomes
\begin{eqnarray}
& & \zeta_{t} + (\eta u)_{x} =0, \nonumber \\
& & u_{t} + \zeta_{x} + \ve uu_{x} - \frac{\sigma^{2}}{3} u_{xxt}=O(\sigma^{4}), \nonumber
\end{eqnarray}
i.e. reduces (if the right-hand side of the second equation is replaced by zero), to the `classical' Boussinesq system,
\cite{W}, which has a linear dispersive term in contrast to the nonlinear dispersive terms present in (\ref{eq12}).
(If the dispersive terms are omitted altogether, the system reduces to the shallow water equations.) Since
it is valid for $\ve=O(1)$, the Serre system, when written in its variable-bottom topography form, has been
found suitable for the description of nonlinear dispersive waves even of larger amplitude, such as water waves
in the near-shore zone before they break. \par
The Cauchy problem for the Serre system in nondimensional variables, that we still denote by $x$, $t$, $u$, 
$\eta=1+\zeta$, is written for $x\in \mathbb{R}$, $t\geq 0$ as
\begin{eqnarray}
& & \eta_{t} + (\eta u)_{x} =0, \label{eq13}\\
& & u_{t} + \eta_{x} + uu_{x} - \frac{1}{3\eta} \bigl[ \eta^{3}(u_{xt} + uu_{xx} - u_{x}^{2})\bigr]_{x}=0, \label{eq14}
\end{eqnarray}
with given initial conditions 
\begin{equation}
\eta(x,0)=\eta_{0}(x), \quad u(x,0)=u_{0}(x), \quad x\in \mathbb{R}. \label{eq15}
\end{equation}
In \cite{Li2} Li proved that the initial-value problem (\ref{eq13})-(\ref{eq15}) is well posed locally in time for 
$(\eta,u)\in H^{s}\times H^{s+1}$, for $s > 3/2$, provided $\min_{x\in \mathbb{R}}\eta_{0}(x)>0$, and that the
property $\min_{x\in \mathbb{R}}\eta(x,t)>0$ is preserved while the solution exists. 
(Here $H^{s}=H^{s}(\mathbb{R})$, for $s$
real, is the subspace of $L^{2}(\mathbb{R})$ consisting of (classes of) functions $f$ for which 
$\int_{-\infty}^{\infty}(1+\xi^{2})^{s} |\hat{f}(\xi)|^{2}d\xi<\infty$, 
where $\hat{f}$ is the Fourier transform of $f$.) Li also provided a rigorous justification for the Serre equations as an
approximation of the Euler equations. Local well-posedness of the system in 1D in its variable bottom formulation 
was proved in \cite{I}. For results on the well-posedness and justification of the general 2D Green--Naghdi equations 
with bottom topography, we refer the reader to \cite{L} and its references. It should be noted that local temporal
existence of the Cauchy problem for the scaled equations (\ref{eq11})-(\ref{eq12}) may be established in intervals of the form $[0,T_{\ve}]$,
where $T_{\ve}=O(1/\ve)$.\par
It is not hard to see, cf. \cite{GN1}, \cite{Li1}, that suitably smooth and decaying solutions of (\ref{eq13})-(\ref{eq15})
preserve, over their temporal interval of existence, the mass $\int_{-\infty}^{\infty}\eta dx$, momentum
$\int_{-\infty}^{\infty}\eta udx$, and energy integrals. The latter invariant (Hamiltonian) is given by 
\begin{equation}
E = \frac{1}{2}\int_{-\infty}^{\infty} \bigl[ \eta u^{2} + \frac{1}{3}\eta^{3}u_{x}^{2} + (\eta - 1)^{2}\bigr] dx. \label{eq16}
\end{equation}
In addition, as Serre had already noted in the second part of his paper, \cite{S}, the system (\ref{eq13})-(\ref{eq14}) 
possesses solitary-wave solutions and a family of periodic (cnoidal) travelling wave solutions; cf. also \cite{CC} for the 
latter. Closed-form formulas are known for both of these families of solutions.\par
In recent years many papers dealing with the numerical solution of the Serre system and its enhanced 
dispersion and variable bottom topography variants have appeared. In these works the reader may find, among other, numerical studies
of the generation, propagation, and interaction of solitary and cnoidal waves, of the interaction of waves with boundaries,
and of the effects of bottom topography on the propagation of the waves. The numerical methods used include spectral 
schemes, cf. e.g. \cite{LHC}, \cite{DCMM}, finite difference and finite volume methods, 
cf. e.g. the early paper \cite{MS}, and \cite{SRT}, \cite{B}, \cite{BCLMT}, 
\cite{BBCCLMT} and its references, \cite{C}, \cite{DCMM}, standard Galerkin methods, cf. e.g.
 \cite{MID}, \cite{MDC}, 
{\em et al}.  In some of these papers the results of numerical simulations with the Serre 
systems have been compared with experimental data and also with numerical solutions of the 
Euler equations. These comparisons bear out the 
effectiveness of the Serre systems in approximating the Euler equations in  a variety of 
variable-bottom-topography test problems, cf.e.g. \cite{SRT},\cite{B},\cite{C}, 
especially when the equations are solved with hybrid
numerical techniques, wherein the advective terms of the equations are discretized by 
shock-capturing techniques while the dispersive terms are treated e.g. by finite differences,
cf. e.g. \cite{BBCCLMT}, \cite{BCLMT}.   
\par  
In the paper at hand we consider the periodic initial-value problem for the Serre 
equations (\ref{eq13})-(\ref{eq14}) with periodic initial data on the spatial interval 
$[0,1]$, assuming  that it has smooth solutions over a temporal
interval $[0,T]$ that satisfy $\min_{(x,t)\in [0,1]\times[0,T]}\eta(x,t)>0$. \par
In section 2 we discretize the problem in space by the standard Galerkin method using the 
smooth periodic splines of order $r\geq 3$ (i.e. piecewise polynomials of degree $r-1\geq 2$) 
on a uniform mesh of meshlength $h$. We compare the Galerkin semidiscrete approximation with 
a suitable spline quasiinterpolant, \cite{TW}, and, using the high order of accuracy of 
the truncation error (due to cancellations resulting from periodicity and the uniform mesh), and an
energy stability and convergence argument, we prove {\em a priori} optimal-order error estimates 
in $L^{2}$, i.e. of $O(h^{r})$, for both components of the semidiscrete solution. This is the first error 
estimate for a numerical method for the Serre system that we are aware of. As expected, the 
presence of the nonlinear dispersive terms complicates the error analysis that is now 
considerably more technical than in analogous proofs of convergence in the case
of Boussinesq systems, \cite{ADM}, and the shallow water equations, \cite{AD}. \par
In section 3 we present the results of numerical experiments that we performed in order 
to approximate solutions of the periodic initial-value problem for the Serre equations using 
mainly cubic
splines in space and the fourth-order accurate, explicit, `classical' Runge-Kutta scheme for
time stepping. We check first that the resulting fully discrete scheme is stable under a Courant
number restriction, enjoys optimal order of accuracy in various norms, and approximates to
high accuracy various types of solutions of the equations including solitary-wave solutions. 
We then use this scheme to illustrate properties of the solitary waves. In the preliminary
section $3.1$ we compare by analytical and numerical means the amplitudes $A_{S}$, $A_{CB}$,
$A_{Euler}$ of the solitary waves of, respectively, the Serre equations, the CB system, and the
Euler equations, corresponding to the same speed $c > 1$, for small values of $c^{2} - 1$.
Our study complements the analogous numerical computations of Li {\emph{et al.}}, \cite{LHC}, and
our conclusion is that always $A_{S} < A_{CB}$ and that 
up to about $c = 1.2$, $A_{S} < A_{Euler} < A_{CB}$ and 
$|A_{Euler} - A_{S}| < |A_{Euler} - A_{CB}|$. For larger speeds the solitary waves of both 
long-wave models are no longer accurate approximations of the solitary wave of the Euler 
equations. In section $3.2$ we study numerically the {\emph{resolution}} of general initial
profiles into sequences of solitary waves when the evolution occurs according to the Serre or
the CB equations. The number of the emerging solitary waves seems to be the same for both systems
and agrees with the prediction of the asymptotic analysis of \cite{EGS}. However, the emerging
solitary waves of the CB equations are faster and of larger amplitude than their Serre counterparts.
Finally, in section $3.3$ we make a careful numerical study of {\emph{overtaking collisions}}
of two solitary waves of the Serre equations, as the ratio of their amplitudes is varied.
We observed types of interaction that are similar to the cases $(a)$, $(b)$, and $(c)$ of
Lax's Lemma $2.3$ in \cite{Lx} for the KdV equation. In addition, for the Serre system, there is
apparently another type of interaction, intermediate between Lax's cases $(a)$ and $(b)$. \par
In this paper we denote, for integer $k\geq 0$, by $H_{per}^{k}=H_{per}^{k}(0,1)$ the usual 
$L^{2}$-based 
Sobolev spaces of periodic functions on $[0,1]$ and their norms by $\|\cdot\|_{k}$. We let 
$C_{per}^{k}=C_{per}^{k}[0,1]$ be the $k$-times continuously differentiable 1-periodic functions. 
The inner
product on $L^{2}=L^{2}(0,1)$ is denoted by $(\cdot,\cdot)$ and the corresponding norm simply by 
$\|\cdot\|$. The norms on $W_{\infty}^{k}=W_{\infty}^{k}(0,1)$ and $L^{\infty}=L^{\infty}(0,1)$ are 
denoted by $\|\cdot\|_{k,\infty}$ and $\|\cdot\|_{\infty}$, respectively. $\mathbb{P}_{r}$  are
the polynomials of degree at most $r$. \par
The paper is dedicated to Jerry Bona, long-time friend, teacher and mentor, on the occasion 
of his 70$^{th}$ birthday.  
\vspace{3pt}  \\
\indent 
{\bf{Acknowledgement:}} This work was partially supported by the programmatic agreement between 
Research Centers-GSRT 2015-2017 in the framework of the Hellenic Republic - Siemens agreement.  \par 
D. E. Mitsotakis was supported by the Marsden Fund administered by the Royal Society of New Zealand.
\section{Galerkin semidiscretization} 
We shall analyze the Galerkin semidiscrete approximation of the periodic initial-value problem for the Serre system in the
following form. Assuming that $\eta$ is positive, we multiply the pde (\ref{eq14}) by $\eta$ and consider the periodic
initial-value problem for the resulting system. Specifically, given $T>0$, for $t\in [0,T]$ we seek 1-periodic functions
$\eta(\cdot,t)$ and $u(\cdot,t)$ satisfying
\begin{equation}
	\begin{aligned} 
		\begin{aligned}
			& \eta_{t} + (\eta u)_{x} = 0, 
			\\
			&\eta u_{t} + \eta\eta_{x} + \eta uu_{x} -\tfrac{1}{3}\bigl[ \eta^{3}(u_{xt} + uu_{xx} -u_{x}^{2})\bigr]_{x}= 0,  
		\end{aligned}
		& \quad (x,t)\in [0,1]\times[0,T], 
		\\
	\eta(x,0) =\eta_{0}(x), \quad u(x,0)=u_{0}(x), \quad 0 \leq x\leq 1, \,\,\,\,\,\,\, & 
	\end{aligned}
	\tag{S}
	\label{eqs}
\end{equation}
where $\eta_{0}$, $u_{0}$ are given 1-periodic functions. For the purposes of the error estimation we shall assume that
$\eta_{0}$ and $u_{0}$ are smooth enough with $\min_{0\leq x\leq 1}\eta_{0}(x)\geq c_{0}>0$ for some constant $c_{0}$
and that (\ref{eqs}) has a unique sufficiently smooth solution $(\eta,u)$ which is 1-periodic in $x$ for all $t\in [0,T]$
and is such that $\eta(x,t)\geq c_{0}$  for $(x,t)\in [0,1]\times [0,T]$.
\subsection{Smooth periodic splines and the quasiinterpolant}
Let $N$ be a positive integer and $h=1/N$, $x_{i}=ih$, $i=0,1,\ldots,N$. For integer $r\geq 2$ consider the associated
$N$-dimensional space of smooth 1-periodic splines
\[
S_{h} = \{\phi \in C_{per}^{r-2}[0,1] : \phi\big|_{[x_{i-1},x_{i}]} \in \mathbb{P}_{r-1}\,, 1\leq i\leq N\}.
\]
It is  well known that $S_{h}$ has the following approximation properties: Given a sufficiently smooth 1-periodic
function $v$, there exists $\chi\in S_{h}$ such that
\[
\sum_{j=0}^{s-1}h^{j}\|v-\chi\|_{j} \leq C h^{s} \|v\|_{s}, \quad 1\leq s\leq r,
\]
and
\[
\sum_{j=0}^{s-1}h^{j}\|v-\chi\|_{j,\infty} \leq C h^{s} \|v\|_{s,\infty}, \quad 1\leq s\leq r,
\]
for some constant $C$ independent of $h$ and $v$. Moreover there exists a constant $C$ independent of $h$ 
such that the inverse properties
\begin{align*}
\|\chi\|_{\beta} & \leq Ch^{-(\beta - \alpha)} \|\chi\|_{\alpha}, \quad 0\leq \alpha\leq \beta\leq r-1, \\
\|\chi\|_{s,\infty} & \leq Ch^{-(s +1/2)} \|\chi\|, \quad 0\leq s \leq r-1,
\end{align*}
hold for all $\chi\in S_{h}$. (In the sequel we shall denote by $C$ generic constants independent of $h$.) \par
Thom\'ee and Wendroff, \cite{TW}, proved that there exists a basis $\{\phi_{j}\}_{j=1}^{N}$ of $S_{h}$ with
$\supp(\phi_{j})=O(h)$, such that if $v$ a sufficiently smooth 1-periodic function, the associated {\em quasiinterpolant}
$Q_{h}v=\sum_{j=1}^{N}v(x_{j})\phi_{j}$ satisfies
\begin{equation}
\|Q_{h}v - v\| \leq Ch^{r} \|v^{(r)}\|. 
\label{eq21}
\end{equation}
In addition, it was shown in \cite{TW} that the basis $\{\phi_{j}\}_{j=1}^{N}$ may be chosen so that the following 
properties hold: \\
(i)\, If $\psi\in S_{h}$, then
\begin{equation}
\|\psi\| \leq Ch^{-1} \max_{1\leq i\leq N}|(\psi,\phi_{i})|.
\label{eq22}
\end{equation}
(It follows from (\ref{eq22}) that if $\psi\in S_{h}$, $f\in L^{2}$ are such that 
\[
(\psi,\phi_{i}) = (f,\phi_{i}) + O(h^{\alpha}), \quad \mbox{for} \quad 1\leq i\leq N,
\]
i.e. if $|(\psi - P_{h}f,\phi_{i})| \leq Ch^{\alpha}$, $1\leq i\leq N$, where $P_{h}$ is the $L^{2}$-projection operator on 
$S_{h}$, then $\|\psi\|\leq \|\psi - P_{h}f\| + \|P_{h}f\| \leq Ch^{\alpha-1} + \|f\|$.) \\
(ii)\, Let $w$ be a sufficiently smooth 1-periodic function and $\nu$, $\kappa$ integers such that 
$0\leq\nu , \kappa\leq r-1$. Then
\begin{equation}
\bigl( (Q_{h}w)^{(\nu)},\phi_{i}^{(\kappa)}\bigr) = (-1)^{\kappa} h w^{(\nu+\kappa)}(x_{i}) + O(h^{2r+j-\nu-\kappa}), \quad
1\leq i\leq N,
\label{eq23}
\end{equation}  
where $j=1$ if $\nu + \kappa$ is even and $j=2$ if $\nu+\kappa$ is odd. \\
(iii)\, Let $f$, $g$ be sufficiently smooth 1-periodic functions and $\nu$ and $\kappa$ as in (ii) above. Let
\[
\beta_{i} = \bigl( f(Q_{h}g)^{(\nu)}, \phi_{i}^{(\kappa)}\bigr) -(-1)^{\kappa} \bigl(Q_{h}\bigl[(fg^{(\nu)})^{(\kappa)}\bigr],\phi_{i}\bigr),
\quad 1\leq i\leq N.
\]
Then
\begin{equation}
\max_{1\leq i\leq N}|\beta_{i}| = O(h^{2r+j-\nu-\kappa}), 
\label{eq24}
\end{equation}
where $j$ as in (ii). \par
It is straightforward to see that the following result also holds for the quasiinterpolant:
\begin{lemma} Let $r\geq 3$, $v\in H_{per}^{r}(0,1)\cap W_{\infty}^{r}(0,1)$ and put \,$V=Q_{h}v$. Then
\begin{align}
& \|V - v\|_{j} \leq Ch^{r-j}\|v\|_{r}, \quad j=0,1,2,
\label{eq25} \\
& \|V - v\|_{j,\infty} \leq Ch^{r-j-1/2}\|v\|_{r,\infty}, \quad j=0,1,2,
\label{eq26} \\
& \|V\|_{j} \leq C, \quad \mbox{and} \quad \|V\|_{j,\infty} \leq C, \quad j=0,1,2.
\label{eq27}
\end{align}
If in addition $\min_{0\leq x\leq 1}v(x)\geq c_{0}>0$, then there exists $h_{0}$ such that
\begin{equation}
\min_{0\leq x\leq 1}V(x) \geq c_{0}/2, \quad \mbox{for} \quad h \leq h_{0}.
\label{eq28}
\end{equation}
\end{lemma}
\begin{proof} The estimates (\ref{eq25}), (\ref{eq26}) follow, as was remarked in \cite{DK}, from the approximation
and inverse properties of $S_{h}$ and (\ref{eq21}), and imply the bounds in (\ref{eq27}). To prove (\ref{eq28}) note
that by (\ref{eq26})
\[
V(x) = \bigl( V(x) - v(x)\bigr) + v(x) \geq -Ch^{r-1/2}\|v\|_{r,\infty} + c_{0} \geq \tfrac{c_{0}}{2},
\]
for $h\leq h_{0}$, with $h_{0}$ such that $Ch_{0}^{r-1/2}\|v\|_{r,\infty}\leq c_{0}/2$. \,
\end{proof}		
\subsection{Consistency of the semidiscrete approximation}
The standard Ga\-ler\-kin semidiscretization of (\ref{eqs}) is defined as follows. We seek  $(\eta_{h},u_{h}) : [0,T]\to S_{h}$,
satisfying for $t\in [0,T]$ the equations
\begin{equation}
\begin{aligned}
& (\eta_{ht},\phi) + ((\eta_{h}u_{h})_{x},\phi) = 0, \quad \forall \phi \in S_{h},\\
& (\eta_{h}u_{ht},\chi) + \tfrac{1}{3}(\eta_{h}^{3}u_{htx},\chi') + (\eta_{h}\eta_{hx},\chi) + (\eta_{h}u_{h}u_{hx},\chi) \\
& \hspace{80pt}  +\tfrac{1}{3}\bigl(\eta_{h}^{3}(u_{h}u_{hxx}-u_{hx}^{2}\bigr),\chi')=0, \quad \forall \chi \in S_{h},
\end{aligned}
\label{eq29}
\end{equation} 
with initial conditions
\begin{equation}
\eta_{h}(0) = Q_{h}\eta_{0}, \quad u_{h}(0) = Q_{h}u_{0}. 
\label{eq210}
\end{equation}
\indent
We first establish the consistency of this semidiscretization to the p.d.e. system in (\ref{eqs}) by proving an optimal-order
$L^{2}$ estimate of a suitable truncation error of (\ref{eq29}).
\begin{proposition}
Let $(\eta,u)$ be the solution of $(\ref{eqs})$  and let $\eta$, $u$ be sufficiently smooth,
1-periodic in $x$. Let $r\geq 3$,
$H=Q_{h}\eta$, $U=Q_{h}u$, and define $\psi, \delta : [0,T]\to S_{h}$ by the equations
\begin{equation}
\begin{aligned}
& (H_{t},\phi) + ((HU)_{x},\phi) = (\psi,\phi), \quad \forall \phi \in S_{h},\\
& (HU_{t},\chi) + \tfrac{1}{3}(H^{3}U_{tx},\chi') + (HH_{x},\chi) + (HUU_{x},\chi) \\
& \hspace{80pt}  +\tfrac{1}{3}\bigl(H^{3}(UU_{xx}-U_{x}^{2}\bigr),\chi')=A(\delta,\chi), 
\quad \forall \chi \in S_{h},
\end{aligned}
\label{eq211}
\end{equation} 
where $A(v,w) = (v,w) + (v',w')$ denotes the $H^{1}$ inner product. 
Then, there exists a constant $C$ independent of $h$, such that
\begin{equation}
\max_{0\leq t\leq T}(\|\psi(t)\| + \|\delta(t)\|_{1}) \leq Ch^{r}.
\label{eq212}
\end{equation}
\end{proposition}
\begin{proof} Let $\rho=Q_{h}\eta - \eta = H-\eta$ and $\sigma=Q_{h}u-u = U-u$. From the first pde in (\ref{eqs}) and
from (\ref{eq211}) we obtain
\[
(\psi,\phi) = (\rho_{t},\phi) + ((HU)_{x} - (\eta u)_{x},\phi), \quad \forall \phi\in S_{h}.
\]
Since
\[
(HU - \eta u)_{x} = \bigl((\rho + \eta)(\sigma+u)\bigr)_{x} - (\eta u)_{x}=(\rho\sigma)_{x} + (\eta\sigma)_{x} + (u\rho)_{x},
\]
it follows that
\begin{equation}
(\psi,\phi) = (\rho_{t},\phi) +((\rho\sigma)_{x},\phi) + (\eta_{x}\sigma+u_{x}\rho,\phi) + (\wt{\psi},\phi), \quad \forall \phi\in S_{h},
\label{eq213}
\end{equation}
where $\wt{\psi} : [0,T]\to S_{h}$ is given by
\[
(\wt{\psi},\phi) = (\eta\sigma_{x},\phi) + (u\rho_{x},\phi), \quad \forall \phi\in S_{h}.
\]
In order to estimate $\wt{\psi}$ we take into account (\ref{eq24}) and obtain for $1\leq i\leq N$ 
\begin{align*}
(\wt{\psi},\phi_{i}) & = (\eta\sigma_{x},\phi_{i}) + (u\rho_{x},\phi_{i}) \\
& = (\eta(Q_{h}u)_{x} - \eta u_{x},\phi_{i}) + (u(Q_{h}\eta)_{x} - u\eta_{x},\phi_{i}) \\
& = (Q_{h}(\eta u_{x}) - \eta u_{x},\phi_{i}) + (Q_{h}(u\eta_{x}) - u\eta_{x},\phi_{i}) + \gamma_{i},
\end{align*}
where $\max_{1\leq i\leq N}|\gamma_{i}|\leq Ch^{2r+1}$. Therefore, using the remark following (\ref{eq22}) and 
(\ref{eq21}) we conclude that	 
\begin{equation}
\|\wt{\psi}\| \leq Ch^{r}.
\label{eq214}
\end{equation}
Taking now $\phi=\psi$ in (\ref{eq213}), by (\ref{eq25}), (\ref{eq214}) we obtain
\begin{equation}
\|\psi\|\leq Ch^{r}.
\label{eq215}
\end{equation}
Proving an analogous estimate for $\delta$ is more complicated due to the presence of the 
nonlinear dispersive terms. From the second pde in (\ref{eqs}) and (\ref{eq211}) we see that
\small
\begin{equation}
\begin{aligned}
A(\delta,\chi) & = (HU_{t} - \eta u_{t},\chi) + \tfrac{1}{3}(H^{3}U_{tx} - \eta^{3}u_{tx},\chi') + (HH_{x}-\eta\eta_{x},\chi)\\
&\,\,\,\,\,\, + (HUU_{x}-\eta uu_{x},\chi) 
+ \tfrac{1}{3}(H^{3}UU_{xx} - \eta^{3}uu_{xx},\chi') \\
&\,\,\,\,\,\, - \tfrac{1}{3}(H^{3}U_{x}^{2}-\eta^{3}u_{x}^{2},\chi'), \quad
\forall \chi \in S_{h}.
\end{aligned}
\label{eq216}
\end{equation}
\normalsize
For the first term in the right-hand side of (\ref{eq216}) we have 
\[
HU_{t} - \eta u_{t} = (\rho + \eta)(\sigma_{t} + u_{t}) - \eta u_{t} = H\sigma_{t} + u_{t}\rho,
\]
and by (\ref{eq21}) and (\ref{eq27}) we get 
\begin{equation}
\|HU_{t} - \eta u_{t}\| = \|H\sigma_{t} + u_{t}\rho\| \leq Ch^{r}.
\label{eq217}
\end{equation}
To treat the second term in the right-hand side of (\ref{eq216}) we write
\begin{align*}
H^{3}U_{tx} - \eta^{3}u_{tx} & =(\rho+\eta)^{3}(\sigma_{tx} + u_{tx}) - \eta^{3}u_{tx} \\
& = U_{tx}\rho^{3} + 3\eta U_{tx}\rho^{2} + 3\eta^{2}\rho\sigma_{tx} + 3\eta^{2}u_{tx}\rho + \eta^{3}\sigma_{tx},
\end{align*}
i.e. 
\begin{equation}
\begin{aligned}
& \tfrac{1}{3}(H^{3}U_{tx} - \eta^{3}u_{tx}) = v_{1} + \wt{v}_{1},  \\
& v_{1} = \tfrac{1}{3}U_{tx}\rho^{3} + \eta U_{tx}\rho^{2} + \eta^{2}\rho\sigma_{tx},\quad 
\wt{v}_{1} = \eta^{2}u_{tx}\rho + \tfrac{1}{3}\eta^{3}\sigma_{tx},
\end{aligned}
\label{eq218}
\end{equation}
Using (\ref{eq27}), (\ref{eq26}), and (\ref{eq25}) we see that 
\begin{align*}
\|v_{1}\| & \leq C(\|\rho^{3}\| + \|\rho^{2}\| + \|\rho\sigma_{tx}\|) \\
& \leq C(\|\rho\|_{\infty}^{2} \|\rho\| + \|\rho\|_{\infty} \|\rho\| + \|\rho\|_{\infty} \|\sigma_{tx}\|) \leq Ch^{2r-3/2},
\end{align*}
from which it follows that
\begin{equation}
\|v_{1}\| \leq Ch^{r}.
\label{eq219}
\end{equation}
For the third term we have 
\[
HH_{x} - \eta\eta_{x} = (\rho + \eta)(\rho_{x} + \eta_{x}) - \eta\eta_{x} = \rho\rho_{x} + \eta_{x}\rho + \eta\rho_{x},
\]
i.e.
\begin{equation}
\begin{aligned}
& HH_{x} - \eta\eta_{x} = v_{2} + \wt{v}_{2}, \\
& v_{2} = \rho\rho_{x} + \eta_{x}\rho, \quad \wt{v}_{2} = \eta\rho_{x},
\end{aligned}
\label{eq220}
\end{equation}
while for the fourth term we write 
\begin{align*}
HUU_{x} - \eta uu_{x} & = (\rho+\eta)(\sigma+u)(\sigma_{x}+u_{x}) - \eta uu_{x} \\
& =(\rho+\eta)(\sigma\sigma_{x} +(u\sigma)_{x} + uu_{x}) - \eta uu_{x}\\
& = H\sigma\sigma_{x} + \rho(u\sigma)_{x} + uu_{x}\rho + \eta u_{x}\sigma + \eta u \sigma_{x},
\end{align*}
or
\begin{equation}
\begin{aligned}
& HUU_{x} - \eta uu_{x} = v_{3} + \wt{v}_{3}, \\
& v_{3} = H\sigma\sigma_{x} + \rho(u\sigma)_{x} + uu_{x}\rho + \eta u_{x}\sigma, \quad \wt{v}_{3} = \eta u\sigma_{x},
\end{aligned}
\label{eq221}
\end{equation}
Using again (\ref{eq25})-(\ref{eq27}) we see, as in the estimation of $v_{1}$ that
\begin{equation}
\|v_{2}\| + \|v_{3}\| \leq Ch^{r}.
\label{eq222}
\end{equation}
For the fifth term in the right-hand side of (\ref{eq216}) we write
\begin{align*}
H^{3} & UU_{xx} - \eta^{3}uu_{xx}  = (\rho+\eta)^{3}(\sigma+u)(\sigma_{xx}+u_{xx}) - \eta^{3}uu_{xx}\\
& = \rho^{3}UU_{xx} + 3\eta\rho^{2}UU_{xx} + 3\eta^{2}\rho\sigma\sigma_{xx} + 3\eta^{2}\rho(\sigma u_{xx} + u\sigma_{xx})\\
&\,\,\,\,\,\, + 3\eta^{2}\rho uu_{xx} + \eta^{3}\sigma\sigma_{xx} + \eta^{3}\sigma u_{xx} + \eta^{3}u\sigma_{xx},
\end{align*} 
i.e.
\begin{equation}
\begin{aligned}
& \tfrac{1}{3} (H^{3}UU_{xx} - \eta^{3}uu_{xx}) = v_{4} + \wt{v}_{4}, \\
&v_{4} = \tfrac{1}{3}UU_{xx}\rho^{3} + \eta UU_{xx}\rho^{2} +\eta^{2}\rho\sigma\sigma_{xx}+\eta^{2}(u_{xx}\rho\sigma + u\rho\sigma_{xx})
+\tfrac{1}{3}\eta^{3}\sigma\sigma_{xx}\\
&\wt{v}_{4} = \eta^{2}uu_{xx}\rho + \tfrac{1}{3}\eta^{3}u_{xx}\sigma + \tfrac{1}{3}\eta^{3}u\sigma_{xx}.
\end{aligned}
\label{eq223}
\end{equation} 	 
Hence, from (\ref{eq25})-(\ref{eq27}) 
\begin{align*}
\|v_{4}\| \leq C (& \|\rho\|_{\infty}^{2}  \|\rho\| + \|\rho\|_{\infty} \|\rho\| + \|\rho\|_{\infty}\|\sigma\|_{\infty} \|\sigma_{xx}\| \\
& + \|\rho\|_{\infty} \|\sigma\| + \|\rho\|_{\infty}\|\sigma_{xx}\| + \|\sigma\|_{\infty} \|\sigma_{xx}\|) \leq Ch^{2r-5/2}.
\end{align*}
Therefore, since $r\geq 3$, 
\begin{equation}
\|v_{4}\|\leq Ch^{r}.
\label{eq224}
\end{equation}
Finally, for the last term in the right-hand side of (\ref{eq216}) we have
\begin{align*}
H^{3}U_{x}^{2} - \eta^{3}u_{x}^{2} & = (\rho + \eta)^{3}(\sigma_{x} + u_{x})^{2} - \eta^{3}u_{x}^{2} \\
& = \rho^{3}U_{x}^{2} + 3\eta\rho^{2}U_{x}^{2} + 3\eta^{2}\rho (\sigma_{x}^{2} + 2u_{x}\sigma_{x} + u_{x}^{2})
+ \eta^{3}(\sigma_{x}^{2} + 2u_{x}\sigma_{x}),
\end{align*}
i.e.
\begin{equation}
\begin{aligned}
&\tfrac{1}{3}(H^{3}U_{x}^{2} - \eta^{3}u_{x}^{2}) = v_{5} + \wt{v}_{5}, \\
& v_{5} = \tfrac{1}{3}\rho^{3}U_{x}^{2} + \eta\rho^{2}U_{x}^{2} + \eta^{2}\rho\sigma_{x}^{2} + 2\eta^{2}\rho u_{x}\sigma_{x}
+ \tfrac{1}{3} \eta^{3}\sigma_{x}^{2}, \\
& \wt{v}_{5} = \eta^{2}u_{x}^{2}\rho + \tfrac{2}{3}\eta^{3}u_{x}\sigma_{x}.
\end{aligned}
\label{eq225}
\end{equation}
From (\ref{eq25})-(\ref{eq27}) we have as before 
\begin{align*}
\|v_{5}\| \leq C(& \|\rho\|_{\infty}^{2} \|\rho\| + \|\rho\|_{\infty}\|\rho\| + \|\rho\|_{\infty}\|\sigma_{x}\|_{\infty} \|\sigma_{x}\| \\
& + \|\rho\|_{\infty} \|\sigma_{x}\| + \|\sigma_{x}\|_{\infty} \|\sigma_{x}\|) \leq Ch^{2r-5/2},
\end{align*}
which gives, since $r\geq 3$, 
\begin{equation}
\|v_{5}\| \leq C h^{r}.
\label{eq226}
\end{equation}
Hence, from (\ref{eq216}), (\ref{eq218}), (\ref{eq220}), (\ref{eq221}), (\ref{eq223}), and (\ref{eq224}) we have for $\chi \in S_{h}$
\begin{equation}
A(\delta,\chi) = (H\sigma_{t} + u_{t}\rho,\chi) + (v_{1},\chi') + (v_{2} + v_{3},\chi) 
+ (v_{4}-v_{5},\chi') + (\wt{\delta},\chi),
\label{eq227}
\end{equation}
where we have defined $\wt{\delta} : [0,T]\to S_{h}$ by the equation 
\[
(\wt{\delta},\chi) = (\wt{v}_{1},\chi') + (\wt{v}_{2},\chi) + (\wt{v}_{3},\chi) + (\wt{v}_{4},\chi') - (\wt{v}_{5},\chi'), \quad \chi\in S_{h}.
\]
Using the definitions of $\wt{v}_{i}$, $1\leq i\leq 5$, from (\ref{eq218}), (\ref{eq220}), (\ref{eq221}), (\ref{eq223}), and (\ref{eq225}),
we obtain for $\chi\in S_{h}$
\begin{equation}
\begin{aligned}
(\wt{\delta},\chi) & = (\eta^{2}u_{tx}\rho,\chi') + \tfrac{1}{3}(\eta^{3}\sigma_{tx},\chi') + (\eta\rho_{x},\chi)
+ (\eta u\sigma_{x},\chi) + (\eta^{2}uu_{xx}\rho,\chi') \\
&\,\,\,\,\, + \tfrac{1}{3}(\eta^{3}u_{xx}\sigma,\chi') + \tfrac{1}{3}(\eta^{3}u\sigma_{xx},\chi') - (\eta^{2}u_{x}^{2}\rho,\chi')
-\tfrac{2}{3}(\eta^{3}u_{x}\sigma_{x},\chi').
\end{aligned}
\label{eq228}
\end{equation}		
The term $(\wt{\delta},\chi)$ consists, like $(\wt{\psi},\phi)$, of $L^{2}$ inner products of $\rho$, $\sigma$, $\sigma_{t}$ and
their spatial derivatives with $\chi$ or $\chi'$ and with smooth periodic functions as weights. To treat these terms we invoke
again the cancellation property (\ref{eq24}) of the quasiinterpolant and write for $1\leq i\leq N$
\begin{align*}
(\wt{\delta},\phi_{i}) & = (\eta^{2}u_{tx}Q_{h}\eta,\phi_{i}') - (\eta^{3}u_{tx},\phi_{i}') + \tfrac{1}{3}\bigl(\eta^{3}(Q_{h}u_{t})_{x},\phi_{i}'\bigr) 
-\tfrac{1}{3}(\eta^{3}u_{tx},\phi_{i}') \\
&\,\,\,\,\,\, +\bigl(\eta(Q_{h}\eta)_{x},\phi_{i}\bigr) - (\eta\eta_{x},\phi_{i}) + (\eta u(Q_{h}u)_{x},\phi_{i}) - (\eta uu_{x},\phi_{i})\\
&\,\,\,\,\,\, + (\eta^{2}uu_{xx}Q_{h}\eta,\phi_{i}') - (\eta^{3}uu_{xx},\phi_{i}') + \tfrac{1}{3}(\eta^{3}u_{xx}Q_{h}u,\phi_{i}')
- \tfrac{1}{3}(\eta^{3}u_{xx}u,\phi_{i}') \\
&\,\,\,\,\,\, + \tfrac{1}{3}(\eta^{3}u(Q_{h}u)_{xx},\phi_{i}') - \tfrac{1}{3}(\eta^{3}uu_{xx},\phi_{i}') - (\eta^{2}u_{x}^{2}Q_{h}\eta,\phi_{i}')
+ (\eta^{3}u_{x}^{2},\phi_{i}') \\
&\,\,\,\,\, - \tfrac{2}{3}(\eta^{3}u_{x}(Q_{h}u)_{x},\phi_{i}') + \tfrac{2}{3}(\eta^{3}u_{x}^{2},\phi_{i}')\\
& = -\bigl(Q_{h}\bigl[(\eta^{3}u_{tx})_{x}\bigr] - (\eta^{3}u_{tx})_{x},\phi_{i}\bigr) + \gamma_{i}^{(1)} \\
&\,\,\,\,\,\, -\tfrac{1}{3}\bigl(Q_{h}\bigl[(\eta^{3}u_{tx})_{x}\bigr]-(\eta^{3}u_{tx})_{x},\phi_{i}\bigr) + \gamma_{i}^{(2)} 
+ \bigl(Q_{h}(\eta\eta_{x}) - \eta\eta_{x},\phi_{i}\bigr) + \gamma_{i}^{(3)} \\
&\,\,\,\,\,\, + \bigl(Q_{h}(\eta uu_{x})-\eta uu_{x},\phi_{i}\bigr) + \gamma_{i}^{(4)}
-\bigl(Q_{h}[(\eta^{3}uu_{xx})_{x}] - (\eta^{3}uu_{xx})_{x},\phi_{i}\bigr) + \gamma_{i}^{(5)} \\
&\,\,\,\,\,\, -\tfrac{1}{3}\bigl(Q_{h}\bigl[ (\eta^{3}uu_{xx})_{x}\bigr] - (\eta^{3}uu_{xx})_{x},\phi_{i}\bigr) + \gamma_{i}^{(6)}\\
&\,\,\,\,\,\, -\tfrac{1}{3}\bigl(Q_{h}\bigl[ (\eta^{3}uu_{xx})_{x}\bigr] - (\eta^{3}uu_{xx})_{x},\phi_{i}\bigr) + \gamma_{i}^{(7)} \\
&\,\,\,\,\,\, + \bigl(Q_{h}\bigl[(\eta^{3}u_{x}^{2})_{x}\bigr] - (\eta^{3}u_{x}^{2})_{x},\phi_{i}\bigr) + \gamma_{i}^{(8)} \\
&\,\,\,\,\,\, +\tfrac{2}{3}\bigl(Q_{h}\bigl[(\eta^{3}u_{x}^{2})_{x}\bigr] - (\eta^{3}u_{x}^{2})_{x},\phi_{i}\bigr) + \gamma_{i}^{(9)},
\end{align*}
where $\max_{1\leq i\leq N}\bigl(|\gamma_{i}^{(1)}| + |\gamma_{i}^{(3)}| +|\gamma_{i}^{(4)}| + |\gamma_{i}^{(5)}|
+ |\gamma_{i}^{(6)}| + |\gamma_{i}^{(8)}|\bigr) \leq Ch^{2r+1}$, while  
$\max_{1\leq i\leq N}\bigl(|\gamma_{i}^{(2)}| + |\gamma_{i}^{(7)}| +|\gamma_{i}^{(9)}|\bigr) \leq Ch^{2r-1}$. Therefore, by
the remark following (\ref{eq22}) and by (\ref{eq21}) we conclude
\begin{equation}
\|\wt{\delta}\| \leq C(h^{2r-2} + h^{r}) \leq Ch^{r}.
\label{eq229}
\end{equation}
Putting now $\chi = \delta$ in (\ref{eq227}) and using (\ref{eq217}), (\ref{eq219}), 
(\ref{eq222}), (\ref{eq224}), (\ref{eq226}), and (\ref{eq229}) we obtain finally
\begin{equation}
\|\delta\|_{1} \leq Ch^{r},
\label{eq230}
\end{equation}
which together with (\ref{eq215}) gives the desired estimate (\ref{eq212}). 
\end{proof}
\subsection{Error estimate}
We now prove using an energy technique an optimal-order $L^{2}$ estimate for the error of the semidiscrete approximation
defined by the initial-value problem (\ref{eq29})-(\ref{eq210}).
\begin{theorem} Suppose that the solution $(\eta,u)$ of $(\ref{eqs})$ is sufficiently smooth and satisfies 
$\min_{0\leq x\leq 1}\eta(x,t) \geq c_{0}$ for $t\in [0,T]$ for some positive constant $c_{0}$. Suppose that $r\geq 3$ and
that $h$ is sufficiently small. Then, there is a unique solution $(\eta_{h},u_{h})$ of $(\ref{eq29})$-$(\ref{eq210})$ on $[0,T]$, which
satisfies
\begin{equation}
\max_{0\leq t\leq T}(\|\eta(t) - \eta_{h}(t)\| + \|u(t) - u_{h}(t)\|) \leq Ch^{r}.
\label{eq231}
\end{equation}
\end{theorem}
\begin{proof} Clearly the ode initial-value-problem (\ref{eq29})-(\ref{eq210}) has a unique solution locally in $t$. While
this solution exists we let $H=Q_{h}\eta$, $U=Q_{h}u$, $\theta=H-\eta_{h}$, and $\xi=U-u_{h}$. 
Using (\ref{eq29}) and (\ref{eq211}) we have 
\begin{equation}
\begin{aligned}
&(\theta_{t},\phi) + ((HU)_{x} - (\eta_{h}u_{h})_{x},\phi) = (\psi,\phi), 
\quad \forall \phi\in S_{h}, \\
& (HU_{t}-\eta_{h}u_{ht}, \chi) + \tfrac{1}{3}(H^{3}U_{tx}-\eta_{h}^{3}u_{htx},\chi') 
+ (w_{1}+w_{2},\chi)\\
& \,\,\,\,\,\,\,\,\,\,\,\,\,\,\,\,\,\,\,\,\,\,\,\,\,\,\,\,\,\,\,\,\,\,\,\,\,\,\,\,\,\,\,\, 
+\tfrac{1}{3}(w_{3} - w_{4},\chi') = A(\delta,\chi),\quad \forall \chi\in S_{h},
\end{aligned}
\label{eq232}
\end{equation}
where
\begin{align*}
w_{1} & = HH_{x} - \eta_{h}\eta_{hx}, \qquad \qquad \, w_{2} = HUU_{x} - \eta_{h}u_{h}u_{hx},\\
w_{3} & = H^{3}UU_{xx} - \eta_{h}^{3}u_{h}u_{hxx}, 
\quad w_{4} = H^{3}U_{x}^{2} - \eta_{h}^{3}u_{hx}^{3}.
\end{align*}
Since
\[
HU - \eta_{h}u_{h} = H(U-u_{h}) + U(H - \eta_{h}) - (H - \eta_{h})(U - u_{h}),
\]
it follows that
\begin{equation}
HU - \eta_{h}u_{h} = H\xi + U\theta - \theta\xi,
\label{eq233}
\end{equation}
and, consequently, from the first equation in (\ref{eq232}) 
\begin{equation}
(\theta_{t},\phi) + ((H\xi)_{x},\phi) + ((U\theta)_{x},\phi) - ((\theta\xi)_{x},\phi) = (\psi,\phi), \quad \forall \phi\in S_{h}.
\label{eq234}
\end{equation}
From (\ref{eq234}), putting $\phi=\theta$ and using integration by parts, we obtain
\begin{equation}
\tfrac{1}{2}\tfrac{d}{dt}\|\theta\|^{2} + ((H\xi)_{x},\theta) + \tfrac{1}{2}(U_{x}\theta,\theta) - ((\theta\xi)_{x},\theta) = (\psi,\theta).
\label{eq235}
\end{equation}
Now, putting $\chi=\xi$ in the second equation in (\ref{eq232}) we see that 
\small
\begin{equation}
(HU_{t}-\eta_{h}u_{ht},\xi) + \tfrac{1}{3}(H^{3}U_{tx}-\eta_{h}^{3}u_{htx},\xi_{x}) + (w_{1}+w_{2},\xi)
+ \tfrac{1}{3}(w_{3} - w_{4},\xi_{x}) = A(\delta,\xi).
\label{eq236}
\end{equation}
\normalsize
For the first term in the left-hand side of (\ref{eq236}) we have
\[
HU_{t} - \eta_{h}u_{ht} = H(U_{t}-u_{ht}) - (H - \eta_{h})(U_{t} - u_{ht}) + U_{t}(H - \eta_{h}),
\]
that is
\[
HU_{t} - \eta_{h}u_{ht} = H\xi_{t} - \theta\xi_{t} + U_{t}\theta,
\]
and therefore
\begin{equation}
\begin{aligned}
(HU_{t} - \eta_{h}u_{ht},\xi) & = (H\xi_{t},\xi) - (\theta\xi_{t},\xi) + (U_{t}\theta,\xi)  \\
& = \tfrac{1}{2}\tfrac{d}{dt}\int_{0}^{1}H(x,t)\xi^{2}(x,t)dx  - \tfrac{1}{2}(H_{t},\xi^{2}) \\
&\,\,\,\,\,\, -\tfrac{1}{2}\tfrac{d}{dt}\int_{0}^{1}\theta(x,t)\xi^{2}(x,t)dx + \tfrac{1}{2}(\theta_{t},\xi^{2}) + (U_{t}\theta,\xi).
\end{aligned}
\label{eq237}
\end{equation}
For the second term in the left-hand side of (\ref{eq236}) it holds that 
\begin{align*}
H^{3}U_{tx} - \eta_{h}^{3}u_{htx} & = H^{3}(U_{tx}-u_{htx}) + (H^{3}-\eta_{h}^{3})U_{tx} - (H^{3}-\eta_{h}^{3})(U_{tx}-u_{htx})\\
&  = H^{3}\xi_{tx} + (H^{3}-\eta_{h}^{3})U_{tx} - (H^{3}-\eta_{h}^{3})\xi_{tx},
\end{align*}
from which 
\[
(H^{3}U_{tx} - \eta_{h}^{3}u_{htx},\xi_{x}) = (H^{3}\xi_{tx},\xi_{x}) + ((H^{3}-\eta_{h}^{3})U_{tx},\xi_{x}) - ((H^{3}-\eta_{h}^{3})\xi_{tx},\xi_{x}).
\]
Hence,
\begin{equation}
\begin{aligned}
 \tfrac{1}{3}(H^{3}U_{tx} & -\eta_{h}^{3}u_{htx},\xi_{x})  = \tfrac{1}{6}\tfrac{d}{dt}\int_{0}^{1}H^{3}(x,t)\xi_{x}^{2}(x,t)dx 
- \tfrac{1}{2} (H^{2}H_{t},\xi_{x}^{2}) \\
& \,\,\,\,\,\, + \tfrac{1}{3}((H^{3}-\eta_{h}^{3})U_{tx},\xi_{x}) - \tfrac{1}{6}\tfrac{d}{dt}\int_{0}^{1}(H^{3}-\eta_{h}^{3})(x,t)\xi_{x}^{2}(x,t)dx \\ 
& \,\,\,\,\,\, + \tfrac{1}{2}(H^{2}H_{t}-\eta_{h}^{2}\eta_{ht},\xi_{x}^{2}).
\end{aligned}
\label{eq238}
\end{equation}
For $w_{1}$ we have
\begin{align*}
w_{1} = HH_{x} - \eta_{h}\eta_{hx} & = H(H_{x}-\eta_{hx}) + H_{x}(H-\eta_{h})-(H-\eta_{h})(H_{x}-\eta_{hx})\\
& = (H\theta)_{x} - \theta\theta_{x},
\end{align*}
i.e.
\begin{equation}
(w_{1},\xi) = -(H\theta,\xi_{x}) - (\theta\theta_{x},\xi).
\label{eq239}
\end{equation}
Moreover
\[
w_{2} = HUU_{x} - \eta_{h}u_{h}u_{hx} = HU(U_{x} - u_{hx}) + (HU - \eta_{h}u_{h})U_{x}  - (HU - \eta_{h}u_{h})(U_{x} - u_{hx}),
\]
and in view of (\ref{eq233})
\[
w_{2} = HU\xi_{x} + (H\xi + U\theta - \theta\xi)U_{x} - (HU - \eta_{h}u_{h})\xi_{x}.
\]
Thus, integrating by parts,
\begin{equation}
\begin{aligned}
(w_{2},\xi) & = -\tfrac{1}{2}((HU)_{x},\xi^{2}) + (HU_{x},\xi^{2}) + (UU_{x}\theta,\xi) - (U_{x}\theta\xi,\xi) \\
& \,\,\,\,\,\, + \tfrac{1}{2}((HU-\eta_{h}u_{h})_{x},\xi^{2}).
\end{aligned}
\label{eq240}
\end{equation}
Now, since 
\begin{align*}
w_{3} = H^{3}UU_{xx} - \eta_{h}^{3}u_{h}u_{hxx} & = H^{3}UU_{xx} - \eta_{h}^{3}u_{h}(U_{xx} - \xi_{xx})\\
& = \eta_{h}^{3}u_{h}\xi_{xx} + (H^{3}U - \eta_{h}^{3}u_{h})U_{xx},
\end{align*}
and since
\[
H^{3}U - \eta_{h}^{3}u_{h} = H^{3}U - \eta_{h}^{3}(U-\xi) = (H^{3} - \eta_{h}^{3})U + \eta_{h}^{3}\xi,
\]
we get
\[
w_{3} = \eta_{h}^{3}u_{h}\xi_{xx} + (H^{3}-\eta_{h}^{3})UU_{xx} + \eta_{h}^{3}U_{xx}\xi,
\]
and therefore
\begin{equation}
(w_{3},\xi_{x}) = -\tfrac{1}{2}(3\eta_{h}^{2}\eta_{hx}u_{h} + \eta_{h}^{3}u_{hx},\xi_{x}^{2}) + ((H^{3}-\eta_{h}^{3})UU_{xx},\xi_{x})
+ (\eta_{h}^{3}U_{xx}\xi,\xi_{x}).
\label{eq241}
\end{equation}
In addition 
\begin{align*}
w_{4} = H^{3}U_{x}^{2} - \eta_{h}^{3}u_{hx}^{2} & = H^{3}U_{x}^{2} - \eta_{h}^{3}(U_{x}-\xi_{x})^{2} = (H^{3}-\eta_{h}^{3})U_{x}^{2}
+ 2\eta_{h}^{3}U_{x}\xi_{x} - \eta_{h}^{3}\xi_{x}^{2} \\
& = (H^{3} - \eta_{h}^{3})U_{x}^{2} + \eta_{h}^{3}U_{x}\xi_{x} + \eta_{h}^{3}u_{hx}\xi_{x}.
\end{align*}
Hence, from this relation and (\ref{eq241}) it follows that
\begin{equation}
\begin{aligned}
(w_{3}-w_{4},\xi_{x}) & = ((H^{3}-\eta_{h}^{3})(UU_{xx}-U_{x}^{2}),\xi_{x}) 
- \tfrac{3}{2}(\eta_{h}^{2}\eta_{hx}u_{h} + \eta_{h}^{3}u_{hx},\xi_{x}^{2}) 	\\
&\,\,\,\,\,\, + (\eta_{h}^{3}U_{xx}\xi,\xi_{x}) - (\eta_{h}^{3}U_{x},\xi_{x}^{2}).
\end{aligned}
\label{eq242}	
\end{equation}
Noting that
\[
((H\xi)_{x},\theta) - ((\theta\xi)_{x},\theta) - (H\theta,\xi_{x}) - (\theta\theta_{x},\xi) = (H_{x}\xi,\theta),
\]
and taking into account (\ref{eq239}), (\ref{eq237}), (\ref{eq238}), (\ref{eq240}), and (\ref{eq242}), if we add (\ref{eq235}) and (\ref{eq236})
we obtain 
\begin{equation}
\begin{aligned}
\tfrac{1}{2}\tfrac{d}{dt} & \|\theta\|^{2} + \tfrac{1}{2}\tfrac{d}{dt}\int_{0}^{1}[H(x,t)\xi^{2}(x,t) + \tfrac{1}{3}H^{3}(x,t)\xi_{x}^{2}(x,t)]dx\\
& = \tfrac{1}{2}\tfrac{d}{dt}\int_{0}^{1}\theta(x,t)\xi^{2}(x,t)dx 
 + \tfrac{1}{6}\tfrac{d}{dt}\int_{0}^{1}(H^{3}-\eta_{h}^{3})(x,t)\xi_{x}^{2}(x,t)dx 	\\
&\,\,\,\,\,\, + \wt{w}_{1} + \wt{w}_{2} + \wt{w}_{3} + \wt{w}_{4}, 
\end{aligned}
\label{eq243}	
\end{equation}
where
\begin{align*}
\wt{w}_{1} & = -(H_{x}\xi,\theta) - \tfrac{1}{2}(U_{x}\theta,\theta), \quad \wt{w}_{2}=\tfrac{1}{2}(H_{t},\xi^{2})-\tfrac{1}{2}(\theta_{t},\xi^{2})
- (U_{t}\theta,\xi),\\
\wt{w}_{3} & = \tfrac{1}{2}(H^{2}H_{t},\xi_{x}^{2}) - \tfrac{1}{3}((H^{3}-\eta_{h}^{3})U_{tx},\xi_{x}) 
- \tfrac{1}{2}(H^{2}H_{t}-\eta_{h}^{2}\eta_{ht},\xi_{x}^{2}), \\
\wt{w}_{4} & = -(w_{2},\xi) - \tfrac{1}{3}(w_{3}-w_{4},\xi_{x}) + (\psi,\theta) + A(\delta,\xi). 	
\end{align*}
From (\ref{eq27}) if follows that 
\begin{equation}
|\wt{w}_{1}| \leq C \|\theta\| \|\xi\| + C \|\theta\|^{2}.
\label{eq244}	
\end{equation}
Taking into account (\ref{eq210}), (\ref{eq25}), (\ref{eq26}), a straightforward estimate for $\theta_{t}$ that we get from (\ref{eq234})
with $\phi = \theta$, and arguing by continuity, we conclude that there is a maximal time $t_{h}\in (0,T]$ such that the solution
of (\ref{eq29})-(\ref{eq210}) exists for $0\leq t\leq t_{h}$ and satisfies 
\begin{equation}{\tag{Y}}
\max_{0\leq s\leq t_{h}}(\|\theta_{t}(s)\|_{\infty} + \|\theta(s)\|_{1,\infty} + \|\xi(s)\|_{1,\infty}) \leq 1.
\label{eqy}	
\end{equation}
Then, from (\ref{eqy}) and (\ref{eq27}) it follows for $0\leq t\leq t_{h}$ that
\begin{equation}
|\wt{w}_{2}| \leq C\|\xi\|^{2} + C\|\theta\| \|\xi\|,
\label{eq245}	
\end{equation}
To derive a bound for $\wt{w}_{3}$, note that $H^{3} - \eta_{h}^{3}=\theta(H^{2}+H\eta_{h} + \eta_{h}^{2})$, and therefore that
\[
\wt{w}_{3} = -\tfrac{1}{3}(H^{2}U_{tx}\theta,\xi_{x}) - \tfrac{1}{3}(HU_{tx}\eta_{h}\theta,\xi_{x}) \\
 - \tfrac{1}{3}(U_{tx}\eta_{h}^{2}\theta,\xi_{x}) + \tfrac{1}{2}(\eta_{h}^{2}\eta_{ht},\xi_{x}^{2}). 	
\]
Since $\|\eta_{h}\|_{1,\infty} \leq \|\theta\|_{1,\infty} + \|H\|_{1,\infty}$, it follows from (\ref{eqy}) and (\ref{eq27}) that 
$\|\eta_{h}\|_{1,\infty}\leq C$ for $0\leq t\leq t_{h}$. Similarly, $\|\eta_{ht}\|_{\infty} \leq C$ for $0\leq t\leq t_{h}$. Therefore, using 
again (\ref{eq27}) we conclude for $0\leq t\leq t_{h}$ that
\begin{equation}
|\wt{w}_{3}| \leq C(\|\theta\| \|\xi_{x}\| + \|\xi_{x}\|^{2}).
\label{eq246}	
\end{equation}
For $\wt{w}_{4}$, using the definitions of $w_{2}$, $w_{3}$, $w_{4}$, noting as before that (\ref{eqy}) implies 
$\|u_{h}\|_{1,\infty}\leq C$  for $0\leq t\leq t_{h}$, and using again (\ref{eqy}), (\ref{eq27}), the identity 
$H^{3}-\eta_{h}^{3} = \theta (H^{2} + H\eta_{h} + \eta_{h}^{2})$, and (\ref{eq212}), gives for $0\leq t\leq t_{h}$
\begin{equation}
|\wt{w}_{4}|\leq C[\|\xi\|^{2} + \|\theta\| \|\xi\| + \|\theta\| \|\xi_{x}\| + \|\xi_{x}\|^{2} + \|\xi\| \|\xi_{x}\|
+ h^{r} (\|\theta\| + \|\xi\|_{1})].
\label{eq247}	
\end{equation}
Therefore, from (\ref{eq243}), and taking into account (\ref{eq244})-(\ref{eq247}), we obtain for $0\leq t\leq t_{h}$ that
\begin{align*}
\tfrac{1}{2}\tfrac{d}{dt}\|\theta(t)\|^{2} & + \tfrac{1}{2}\tfrac{d}{dt}\int_{0}^{1}[H(x,t)\xi^{2}(x,t) + H^{3}(x,t)\xi_{x}^{2}(x,t)]dx
\leq 	\tfrac{1}{2}\tfrac{d}{dt}\int_{0}^{1}\theta(x,t)\xi^{2}(x,t)dx \\
& \,\,\,\,\,\, + \tfrac{1}{6}\tfrac{d}{dt}\int_{0}^{1}(H^{3} - \eta_{h}^{3})(x,t)\xi_{x}^{2}(x,t)dx + C(\|\theta(t)\| + \|\xi(t)\|_{1})^{2}
+ Ch^{2r}.
\end{align*}
Integrating both sides of this inequality with respect to $t$ and taking again into account (\ref{eqy}) and noting that 
$\theta(0)=\xi(0)=0$, yields for $t\leq t_{h}$ that
\begin{equation}
\begin{aligned}	
\tfrac{1}{2}\|\theta(t)\|^{2} & + \tfrac{1}{2}\int_{0}^{1}[H(x,t)\xi^{2}(x,t)+\tfrac{1}{3}H^{3}(x,t)\xi_{x}^{2}(x,t)]dx \\
& \,\,\,\,\, \leq C\int_{0}^{t}(\|\theta(s)\|^{2} + \|\xi(s)\|_{1}^{2})ds + Ch^{2r}.
\end{aligned}
\label{eq248}	
\end{equation}
The left-hand side of this inequality is a sort of discrete analog of the Hamiltonian, cf. (\ref{eq16}), appropriate for the
periodic Serre system. From our
hypothesis that $\min_{0\leq x\leq 1}\eta(x,t)\geq c_{0}>0$ for each $t\in [0,T]$ and (\ref{eq28}) of Lemma 2.1, we conclude
from (\ref{eq248}) that
\[
\|\theta(t)\|^{2} + \|\xi(t)\|_{1}^{2} \leq C_{0} \int_{0}^{t}(\|\theta(s)\|^{2} + \|\xi(s)\|_{1}^{2})ds + C_{0}h^{2r},
\]
holds for $0\leq t\leq t_{h}$ for a constant $C_{0}$ independent of $t_{h}$ and $h$. From this relation and
Gronwall's lemma it follows that 
\begin{equation}
\|\theta(t)\| + \|\xi(t)\|_{1} \leq C_{1}h^{r}, \quad \mbox{for} \quad t\leq t_{h},
\label{eq249} 	
\end{equation}
for a constant $C_{1}=C_{1}(T)$ independent of $t_{h}$ and $h$. This inequality and the 
inverse properties of $S_{h}$ yield that
$\|\theta\|_{1,\infty} \leq Ch^{r-3/2}$ and $\|\xi\|_{1,\infty}\leq Ch^{r-1/2}$ 
for $t\leq t_{h}$. In addition, it is straightforward
to see that taking $\phi=\theta_{t}$ in (\ref{eq234})  and using (\ref{eq249})  
gives $\|\theta_{t}\|\leq Ch^{r-1}$ and, therefore, that 
$\|\theta_{t}\|_{\infty}\leq Ch^{r-3/2}$ for $0\leq t\leq t_{h}$. We conclude that 
$\|\theta_{t}\|_{\infty} + \|\theta\|_{1,\infty} + \|\xi\|_{1,\infty} \leq Ch^{r-3/2}$ for 
$0\leq t\leq t_{h}$. This implies, provided $h$ was taken sufficiently small, that $t_{h}$ was 
not maximal in (\ref{eqy}). Hence, we may take $t_{h}=T$, and (\ref{eq231}) follows from 
(\ref{eq249}) and (\ref{eq25}). 
\end{proof}
\subsection{Remarks} (i) It is straightforward to see that the error estimate (\ref{eq231}) still holds if we take any initial 
condition $\eta_{h}(0)\in S_{h}$ in (\ref{eq210}) that satisfies $\|\eta_{h}(0) - \eta_{0}\|\leq Ch^{r}$, e.g. the
$L^{2}$ projection or the interpolant of $\eta_{0}$ on $S_{h}$. Then (\ref{eq21}) implies that $\|\theta(0)\|\leq Ch^{r}$, and 
one may easily check that (\ref{eqy}) is still valid for some $t_{h}\in(0,T]$, and that (\ref{eq248}), and therefore (\ref{eq249}),
still hold. The conclusion of Theorem 2.3 follows. \\
(ii) The error estimate (\ref{eq231}) is still valid if we choose as initial condition $u_{h}(0)$ an `elliptic' projection of $u_{0}$
on $S_{h}$ defined in terms of the bilinear form 
\[
B_{Serre}(\chi,\phi;\eta_{0}):= (\eta_{0}\chi,\phi) + \tfrac{1}{3}(\eta_{0}^{3}\chi',\phi'), \quad \chi,\phi \in S_{h}.
\]
Indeed, if $v_{h}\in S_{h}$ is the unique (since $\eta_{0}\geq c_{0}>0$) function in $S_{h}$ for which 
\begin{equation}
(\eta_{0}v_{h},\phi) + \tfrac{1}{3}(\eta_{0}^{3}v_{h}',\phi') = (\eta_{0}u_{0},\phi) + \tfrac{1}{3}(\eta_{0}^{3}u_{0}',\phi'), \quad 
\forall 	\phi \in S_{h},
\label{eq250}
\end{equation}
then 
\begin{equation}
\|Q_{h}u_{0} - v_{h}\|_{1} \leq Ch^{r},
\label{eq251}	
\end{equation}
i.e. $v_{h}$ is superoptimally close to $Q_{h}u_{0}$ in $H_{per}^{1}$. 
To see this, note that if $\ve_{h}=Q_{h}u_{0} - v_{h}$ and $e=Q_{h}u_{0} - u_{0}$, then 
\begin{equation}
(\eta_{0}\ve_{h},\phi) + \tfrac{1}{3}(\eta_{0}^{3}\ve_{h}',\phi') = (\eta_{0}e,\phi) + \tfrac{1}{3}(\eta_{0}^{3}e',\phi'), \quad 
\forall \phi \in S_{h}.
\label{eq252}	
\end{equation}
Define now $\gamma\in S_{h}$ by the equation
\begin{equation}
(\gamma,\phi) = (\eta_{0}e,\phi) + \tfrac{1}{3}(\eta_{0}^{3}e',\phi'), \quad \forall \phi\in S_{h}.
\label{eq253}	
\end{equation}
Using the properties of the quasiinterpolant, and in particular (\ref{eq24}), 
we have for $1\leq i\leq N$
\[
(\eta_{0}e,\phi_{i}) = (\eta_{0}Q_{h}u_{0}-\eta_{0}u_{0},\phi_{i}) 
= (Q_{h}(\eta_{0}u_{0}) - \eta_{0}u_{0},\phi_{i}) + \beta_{i}^{(1)},
\] 
where $\max_{1\leq i\leq N}|\beta_{i}^{(1)}| \leq Ch^{2r+1}$. Similarly,
\[
(\eta_{0}^{3}e',\phi_{i}') = - \bigl(Q_{h}\bigl[(\eta_{0}^{3}u_{0}')'\bigr]-(\eta_{0}^{3}u_{0}')',\phi_{i}\bigr) + \beta_{i}^{(2)},
\]
where $\max_{1\leq i\leq N}|\beta_{i}^{(2)}| \leq Ch^{2r-1}$. We conclude, in view of (\ref{eq253}), (\ref{eq21}), and the remark
following (\ref{eq22}), that 
\[
\|\gamma\| \leq C(h^{r} + h^{2r-2}) \leq Ch^{r}.
\]
Therefore, since (\ref{eq252}) and (\ref{eq253}) imply
\[
(\eta_{0}\ve_{h},\phi) + \tfrac{1}{3}(\eta_{0}^{3}\ve_{h}',\phi') = (\gamma,\phi), \quad \forall \phi \in S_{h},
\]
putting $\phi = \ve_{h}$ yields $\|\ve_{h}\|_{1} \leq C\|\gamma\| \leq Ch^{r}$, 
i.e. that (\ref{eq251})  is valid. \par
We take now $u_{h}(0) = v_{h}$ as initial condition in (\ref{eq210}). Note that 
(\ref{eqy}) still holds for some $t_{h}\in(0,T]$,
since $\|\xi(0)\|_{1}= \|Q_{h}u_{0} - v_{h}\|_{1} \leq Ch^{r}$ and therefore 
$\|\xi(0)\|_{1,\infty} \leq Ch^{r-1/2}$. Moreover
(\ref{eq248}) and (\ref{eq249})  still hold and the conclusion of Theorem 2.3 follows. From this and the previous remark
it is clear that the usual B-spline basis may be used in the finite element computations, i.e. that there is no need of
computing with the special basis $\{\phi_{i}\}_{i=1}^{N}$.
\section{Numerical experiments} In this section we present the results of some numerical 
experiments that
we performed to approximate solutions of the periodic initial-value problem for the Serre equations
using the standard Galerkin semidiscretization (\ref{eq29}) with the spatial interval taken to be of 
the form $[-L,L]$, so that $h=2L/N$. We generally used cubic splines, i.e. $S_{h}$ with $r=4$, and 
computed the initial values $\eta_{h}(0)$, $u_{h}(0)$ as the $L^{2}$ projections of $\eta_{0}$, $u_{0}$
on $S_{h}$. The semidiscrete initial-value problem was discretized in time by the `classical', explicit,
4$^{th}$-order accurate Runge-Kutta scheme, with a uniform time step denoted by $k$.
This fully discrete scheme was used in simulations of solutions of the Serre system in 
\cite{MID}. Numerical evidence from \cite{MID} and the present work, and also
theoretical and numerical evidence from \cite{AD1} and \cite{AD2} in the case of the
`classical' Boussinesq system, a close relative of the Serre equations, suggests that the 
fully discrete method
under considerations is fourth-order accurate in the temporal variable and stable under a 
Courant number restriction of the form $k/h\leq r_{0}$.\par
We checked the accuracy of the fully discrete scheme by taking as solution of the Serre
system the solitary wave, \cite{S}, given by $\eta_{S}(x-x_{0}-ct)$,
$u_{S}(x - x_{0}-ct)$, where $c > 1$ and
\begin{equation} 
\begin{aligned}
\eta_{S}(\xi) & = 1 + A_{S}\mathrm{sech}^{2}(K\xi),\quad 
u_{S}(\xi) & = c\bigl( 1 - \frac{1}{\eta_{S}(\xi)}\bigr),\\
A_{S}= & \,\,c^{2} - 1, \quad  K =\sqrt{\frac{3A_{S}}{4c^{2}}}\,.	
\end{aligned}
\label{eq31}	
\end{equation}
In the numerical experiments we took $c=1.2$ and integrated on the spatial interval 
$[-150,150]$ up to $T=100$. (The solution is effectively periodic; initially the solitary 
wave was centered at $x_{0}=-100$ and the value of $\eta_{S}$ at $x=\pm150$ differed 
from $1$ by an amount smaller than the machine epsilon.) We took $h = 2L/N=300/N$,
$k=T/M$ and computed with a fixed ratio $k/h=0.1$ for increasing $N$. In the case of
cubic spline space discretization the numerical convergence rates in the $L^{2}$ and
$L^{\infty}$ norms approached four, while those in the $H^{1}$ and $H^{2}$ norms three
and two, respectively. (The order of magnitude of the $L^{2}$ errors for $\eta$ ranged
from $10^{-6}$ for $N = 600$ to $10^{-10}$ for $N = 6000$). The observed optimal-order 
$L^{2}$ and $L^{\infty}$ rates also suggest that there is no temporal order reduction in
the scheme. It should also be noted that for this experiment the relative errors (with
respect to their initial values) of two invariants of the problem, namely the energy
(Hamiltonian) $E$, defined by the analogous to (\ref{eq16}) formula on $[-L,L]$, and
the momentum $I = \int_{-L}^{L}\eta udx$, ranged, at $T=100$, from $O(10^{-7})$ for
$N = 600$ to $O(10^{-12})$ for $N = 6000$ in the case  of $E$, and from $O(10^{-8})$
for $N = 600$ to $O(10^{-13})$ for $N = 3750$ for $I$. (The mass $\int_{-L}^{L}\eta dx$ was
preserved of course to machine epsilon.) We conclude that the outcome of these experiments
confirms that of the analogous computations in \cite{MID}. \par
In the case of quadratic splines the numerical convergence rates in the $L^{2}$ 
and $L^{\infty}$ norms approached three as $N$ increased, while those corresponding to the
$H^{1}$ and $H^{2}$ norms approached two and one, respectively, as expected. The $L^{2}$
errors for $\eta$ ranged from $O(10^{-5})$ for $N = 600$ to $O(10^{-8})$ for $N = 6000$,
while the relative errors of the invariants ranged, for the same values of $N$, from
$O(10^{-6})$ to $O(10^{-10})$ for $E$ and from $O(10^{-7})$ to $O(10^{-12})$ for $I$.
It was also noted that at the spatial nodes $x_{i}$ the errors of the scheme with 
quadratic splines appeared to be $O(h^{4})$, i.e. superconvergent. \par
As noted above, the fully discrete scheme requires for stability a bound on the Courant 
number $k/h$. In this example, in the case of cubic splines the $L^{2}$ errors were of
the same order of magnitude up to about $k/h=1/2$, increased slowly for larger values
of $k/h$, and the scheme became unstable for $k/h\geq 3$. For quadratic splines, the
$L^{2}$ errors preserved their order of magnitude up to $k/h = 1$ and increased slowly
afterwards until violent instability occurred when $k/h\geq 8$. \par
These numerical results, in addition to many other similar ones that we obtained by
integrating with this scheme {\emph{cnoidal-wave}} solutions of the Serre system and also
`artificial' solutions of the nonhomogeneous equations, and in addition to similar 
results of \cite{MID}, confirm the good accuracy and stability of the scheme and give us
confidence in using it to simulate properties of the solitary waves of the Serre
equations in the sequel.
\subsection{Remarks on solitary waves} Since the numerical experiments of the two following
subsections will focus on properties of the solitary waves of the Serre equations, we make
here some observations comparing them to the solitary waves of the `classical' Boussinesq
system (CB) and the Euler equations. \par
By (\ref{eq31}) we have for the solitary wave of the Serre equations with speed $c>1$ that
\begin{equation}
\begin{aligned}
\zeta(\xi) & = A_{S} \mathrm{sech}^{2}(K\xi), \quad \xi \in \mathbb{R},\\
A_{S} & = c^{2}-1, \quad K = \sqrt{3A_{S}/4c^{2}}.
\end{aligned}
\label{eq32}
\end{equation}	
(Note, incidentally, that the speed-amplitude relation, $c = \sqrt{1 + A_{S}}$, coincides
with Scott-Russell's empirical formula.) In the case of CB there exist no closed-form
solutions for the solitary waves but one may easily prove, cf. \cite{AD1}, that the speed of 
the solitary wave of amplitude $A_{CB}$ is given by the formula 
\begin{equation}
c^{2} = \frac{6(1+A_{CB})^{2}}{3 + 2A_{CB}}\cdot 
\frac{(1+A_{CB})\ln(1+A_{CB})-A_{CB}	}{A_{CB}^{2}}.
\label{eq33}
\end{equation}
Letting
\[
f(x) = \frac{6(1+x)^{2}}{3+2x}\cdot\frac{(1+x)\ln(1+x)-x}{x^{2}} - 1, \quad x>0,
\]
one may see by straightforward calculus arguments that $f$ is continuous on $[0,\infty)$, 
$f(0)=0$, and $f$ is monotonically increasing on $[0,\infty)$. Moreover, $f(x)<x$
for $x>0$, from which we infer that for solitary waves of the CB and the Serre systems
of the same speed $c>1$, since $A_{S} = c^{2}-1 = f(A_{CB})$, it holds that
\begin{equation}
A_{S} < A_{CB}.
\label{eq34}	
\end{equation}
Since for $x$ small we have 
$f(x) = x - \tfrac{1}{6}x^{2}+\tfrac{1}{90}x^{3} + \tfrac{7}{240}x^{4} + O(x^{5})$, inverting
the series, we get for $y = c^{2}-1$, $c>1$, $y$ small, that
\begin{equation}
A_{CB} = y + \tfrac{1}{6}y^{2} + \tfrac{2}{45}y^{3} - \tfrac{13}{1080}y^{4} + O(y^{5}).
\label{eq35}	
\end{equation}
The analog of the expansion of $c^{2} - 1$ in terms of the amplitude of the solitary
wave of the Euler equations has been derived e.g. by Long, \cite{Lg}, equation ($46$).
Inverting Long's expansion we obtain for $y = c^{2} -1$, $c>1$, $y$ small,
\begin{equation}
A_{Euler} = y + \tfrac{1}{20}y^{2} + \tfrac{67}{1400}y^{3} 
+ \tfrac{73}{1600}y^{4} + O(y^{5}).	
\label{eq36}
\end{equation}
(To our knowledge, convergence of this expansion has not been proved.) Comparing
(\ref{eq35}) and (\ref{eq36}) with the exact relation
\begin{equation}
A_{S} = y,
\label{eq37}
\end{equation}
we see that for $c$ close to $1$ $A_{S} < A_{Euler} < A_{CB}$ and that
$|A_{Euler} - A_{S}| < |A_{Euler} - A_{CB}|$. These observations are confirmed by 
Figure \ref{fig30} in which the solitary waves 
\begin{figure}[ht!] 
\centering
\vspace{0.4cm}
\includegraphics[scale=0.69]{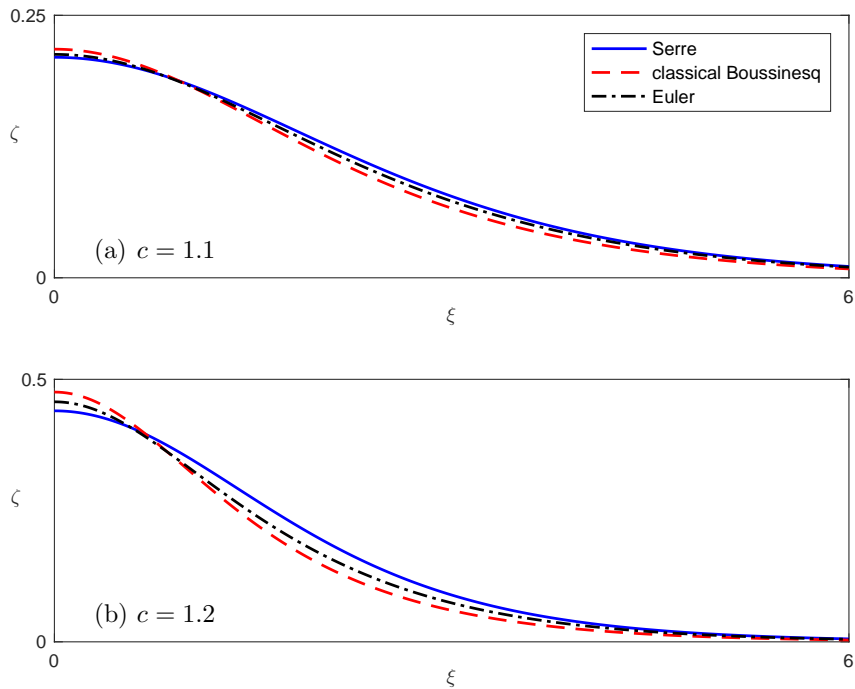}
\caption{Profiles of solitary waves for $\xi=x-ct$ 
(only the right-hand halves of the waves are shown) for the Serre, CB, and Euler equations,
for speeds (a) $c=1.1$, (b) $c = 1.2$.}
\label{fig30}
\end{figure} 
for the Euler equations (computed by the numerical method of \cite{DC}) and the 
`classical' Boussinesq system (computed by a spectral spatial discretization and the
Petviashvili nonlinear system solver as in \cite{AlD}) are compared for $c=1.1$ and $c=1.2$
to the solitary waves of the Serre system. In Figure \ref{fig30} the numerical values
are 
\begin{align*}
A_{S} & = 0.21, \quad A_{Euler} = 0.21276, \quad A_{CB} = 0.21774 
\quad\, \text{for}\quad c=1.1, \\
A_{S} & = 0.44, \quad A_{Euler} = 0.45715, \quad A_{CB} = 0.47573 
\quad\, \text{for}\quad c=1.2.
\end{align*}
Our numerical evidence agrees with the analogous results of Li {\em{et al.}}, 
\cite{LHC}, where a comparative study was made by numerical means of solitary waves of the
Euler equations, the CB system, the KdV equation and the Serre equations (called the
Su-Gardner equations). A general conclusion from \cite{LHC} is that up to speeds (in the
variables of the paper at hand) of about $1.2$ the amplitude and mass, i.e. the integral
$\int\eta dx$, of the Serre solitary wave, are closer to the analogous quantities of the 
solitary wave of the Euler equations than their CB counterparts. For larger speeds
the solitary waves of both long-wave models are not accurate approximations of the 
solitary wave of the Euler equations. 
\subsection{Resolution of initial profiles into solitary waves} It is well known that solitary waves play
a distinguished role in the evolution and long-time behavior of solutions of the initial-value problem for 
the associated 
nonlinear dispersive systems that emanate from initial data that decay suitably fast at infinity. This
property, of the {\em resolution} of general initial data into sequences of solitary waves followed by
slower oscillatory dispersive tails of small amplitude, has been rigorously proved for integrable one-way
models such as the KdV equation and observed numerically in the case of many other examples of nonlinear
dispersive wave equations that possess solitary waves. 
In particular, this property has been observed in
the case of the Serre equations in numerical experiments in \cite{LHC}, \cite{MID}, and \cite{EGS}. \par
In order to complement these studies we integrated the Serre system with our fully discrete scheme 
(using cubic splines for the spatial discretization) on the interval $[-300,300]$, with $h=0.1$ and $k=0.01$.
\begin{figure}[b!] 
\centering
\vspace{0.4cm}
  \includegraphics[scale=0.65]{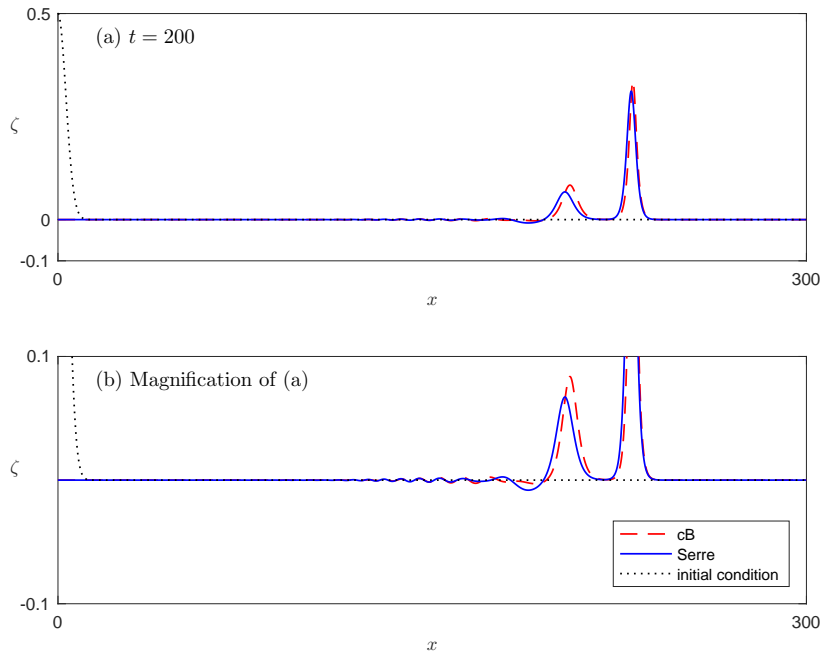}
  \caption{Resolution of a Gaussian into solitary waves ($a=0.5$, $b=0.05$), Serre and CB systems.}
  \label{fig31}
\end{figure} 
We studied the evolution of initial Gaussian profiles of the form $\zeta(x,0)=ae^{-bx^{2}}$, $u(x,0)=0$
(recall that $\eta = 1 + \zeta$), and compared with the analogous evolution of solutions of the `classical'
\begin{figure}[t!]
 \centering
\vspace{0.2cm}
  \includegraphics[scale=0.65]{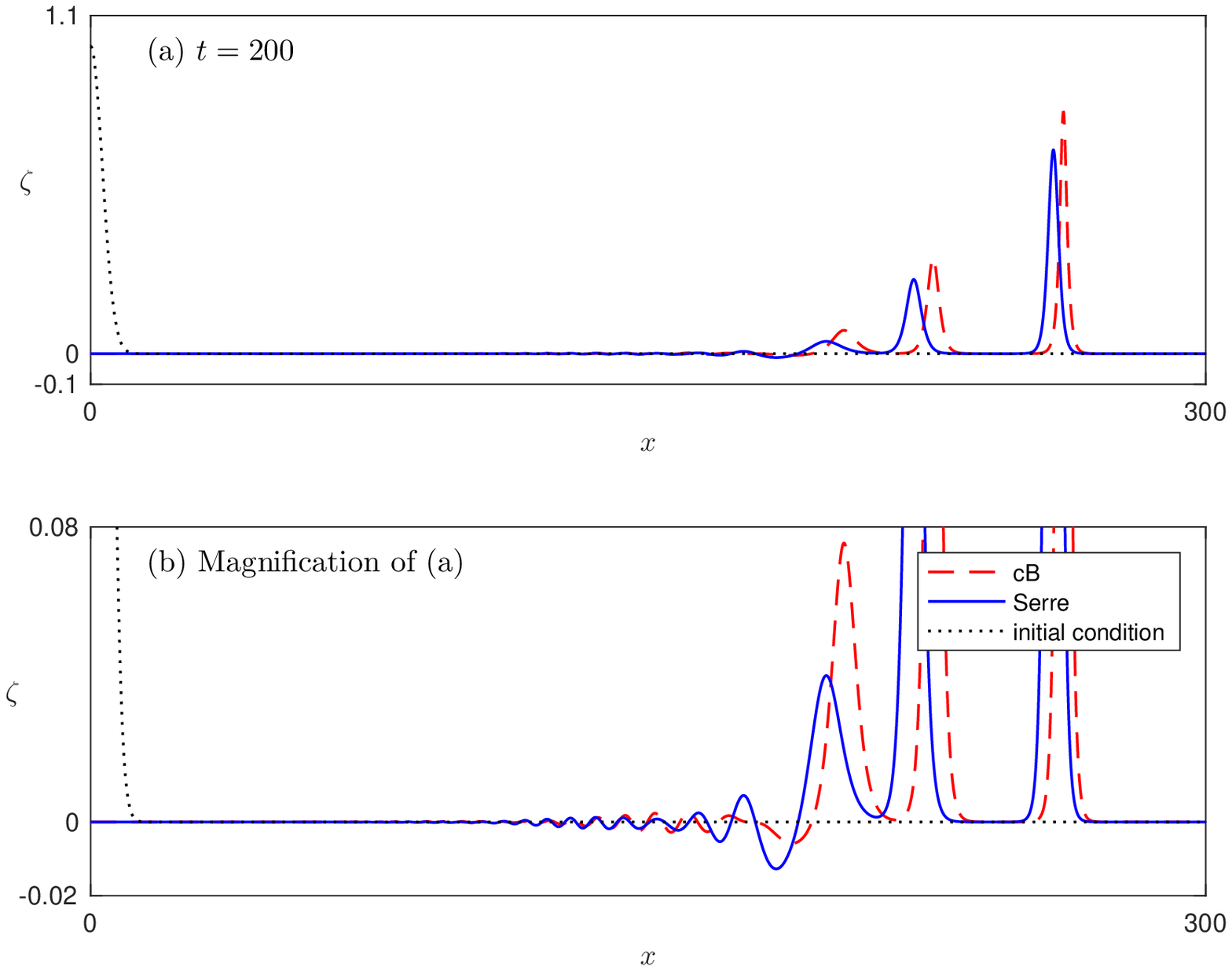}
  \caption{Resolution of a Gaussian into solitary waves ($a=1$, $b=0.05$), Serre and CB systems.}
  \label{fig32}
\end{figure}
\begin{figure}[h!]
  \centering
  \vspace{0.2cm}
  \includegraphics[scale=0.65]{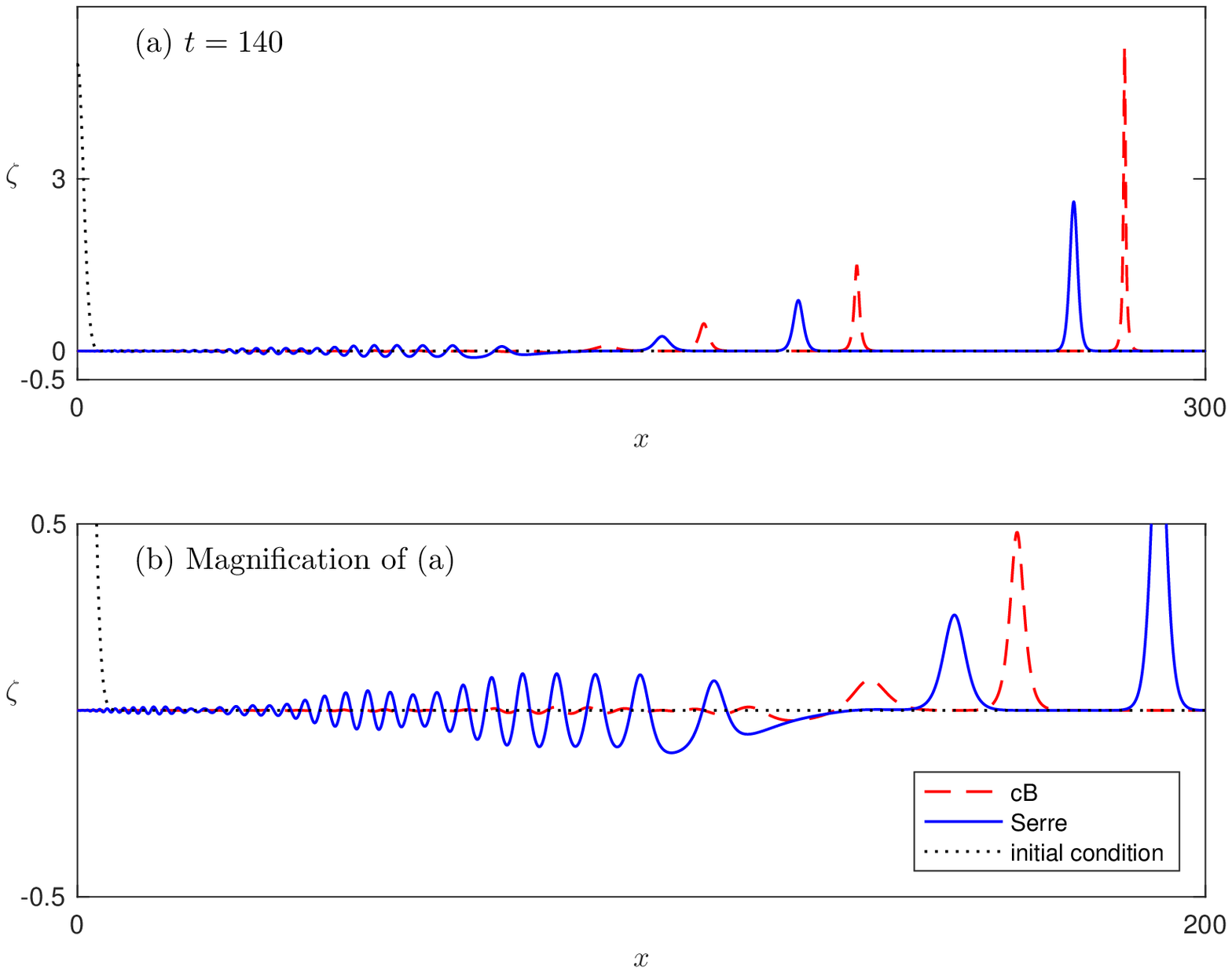}
  \caption{Resolution of a Gaussian into solitary waves ($a=5$, $b=0.2$), Serre and CB systems.}
  \label{fig33}
\end{figure}
Boussinesq system (CB) with the same initial conditions. The initial profile evolves, after
some time,
into two symmetric wavetrains that travel in opposite directions and consist of a number of
solitary waves
plus a trailing dispersive tail. (In Figures \ref{fig31}-\ref{fig33} only 
the rightwards-travelling part of the solution is shown.) 
In Figure \ref{fig31} one may observe the solution emanating from the Gaussian with $a=0.5$, 
$b=0.05$; by $t=200$ two solitary waves have been formed for both systems.
Figures \ref{fig32} and \ref{fig33} show analogous 
resolution profiles produced by 
Gaussians with 
  $a=1$, $b=0.05$, and $a=5$, $b=0.2$, at $t=200$ and $t=140$, respectively.
By these temporal values
 three solitary waves have emerged for both
systems in both cases. (The magnified graph in Figure \ref{fig33}
shows the smallest generated solitary wave pulses of both systems and the base of the next larger solitary
wave of the Serre system, and provides a clear view of the dispersive tails of the two wavetrains.) We 
observe that the emerging
solitary waves of the Serre system have smaller speeds and amplitudes than their companion 
CB solitary waves. The oscillations of the dispersive tail behind the Serre solitary wavetrains are much 
larger in amplitude and number than those of the dispersive tail in the CB case, which is a nonlinear 
feature of the Serre system as the linearized equations of both systems coincide. It might be of interest 
to point out that we ran several experiments with Gaussian initial profiles varying the parameters $a$ and 
$b$; in each case we noted that the number of solitary waves $N_{s}$ produced by about $t=200$ was the same 
for both systems. In particular, the results of Table 1 of \cite{ADM1} apparently hold for the Serre system 
as well and therefore $N_{s}$ appears to be proportional to $\int\sqrt{\eta_{0}-1} \sim (a/b)^{1/2}$ as in 
the case of the KdV equation, \cite{W}, and as predicted by asymptotic analysis for the Serre system in 
\cite{EGS}. 
\subsection{Overtaking collisions of solitary waves} As is well known, when two solitary waves of many 
nonlinear dispersive systems collide, the solitary waves that emerge remain largely unchanged, having in
general slightly different amplitudes and speeds and undergoing small phase shifts. With the exception
of integrable equations, such as the KdV, the collisions are inelastic, producing in addition 
small-amplitude 
dispersive oscillatory tails after the interactions. \par 
In the case of the Serre equations head-on collisions of solitary waves have appeared in the literature,
cf. e.g. \cite{MID} and its references. In this section we shall focus on {\em overtaking collisions}
of solitary waves of different speeds propagating in the same direction. The highly accurate numerical
scheme tested in section 3.1 affords a detailed simulation of this type of 
interactions. Overtaking 
collisions for the Serre system have been previously studied numerically in \cite{MS}, \cite{LHC},
\cite{MID}. For analogous studies in the case of Boussinesq systems cf. e.g. \cite{ADM1}, \cite{AD1}.
A thorough study of collisions of solitary waves, and in particular of overtaking collisions, by
numerical and experimental means has been performed by Craig {\em et al.} \cite{CGHHS} in the case of the
Euler equations. \par 
We studied the overtaking collisions by integrating the Serre equations with our fully discrete method using
cubic splines on the spatial interval $[-400,400]$ taking $h=0.1$ and $k=0.01$. The initial conditions
were two solitary waves of the form (\ref{eq31}), of which the larger one, of amplitude $A_{S}=a_{1}$, was
centered at $x=-30$, while the smaller, of amplitude $a_{2} < a_{1}$, was centered 
at $x=30$. In the sequel we let 
$r=a_{1}/a_{2}$, and, unless otherwise indicated, we took $a_{1}=1$ and varied $a_{2}$. \par 
For small values of $r$, specifically for $r \leq 3.0967$, we observed that during the collision the 
amplitude of the larger solitary wave decreases monotonically while that of the smaller one increases
also monotonically. At the end of the interaction the two waves have exchanged the order of their 
positions and in 
\begin{figure}[ht!]
  \centering
  \vspace{0.4cm}
   \includegraphics[scale=0.67]{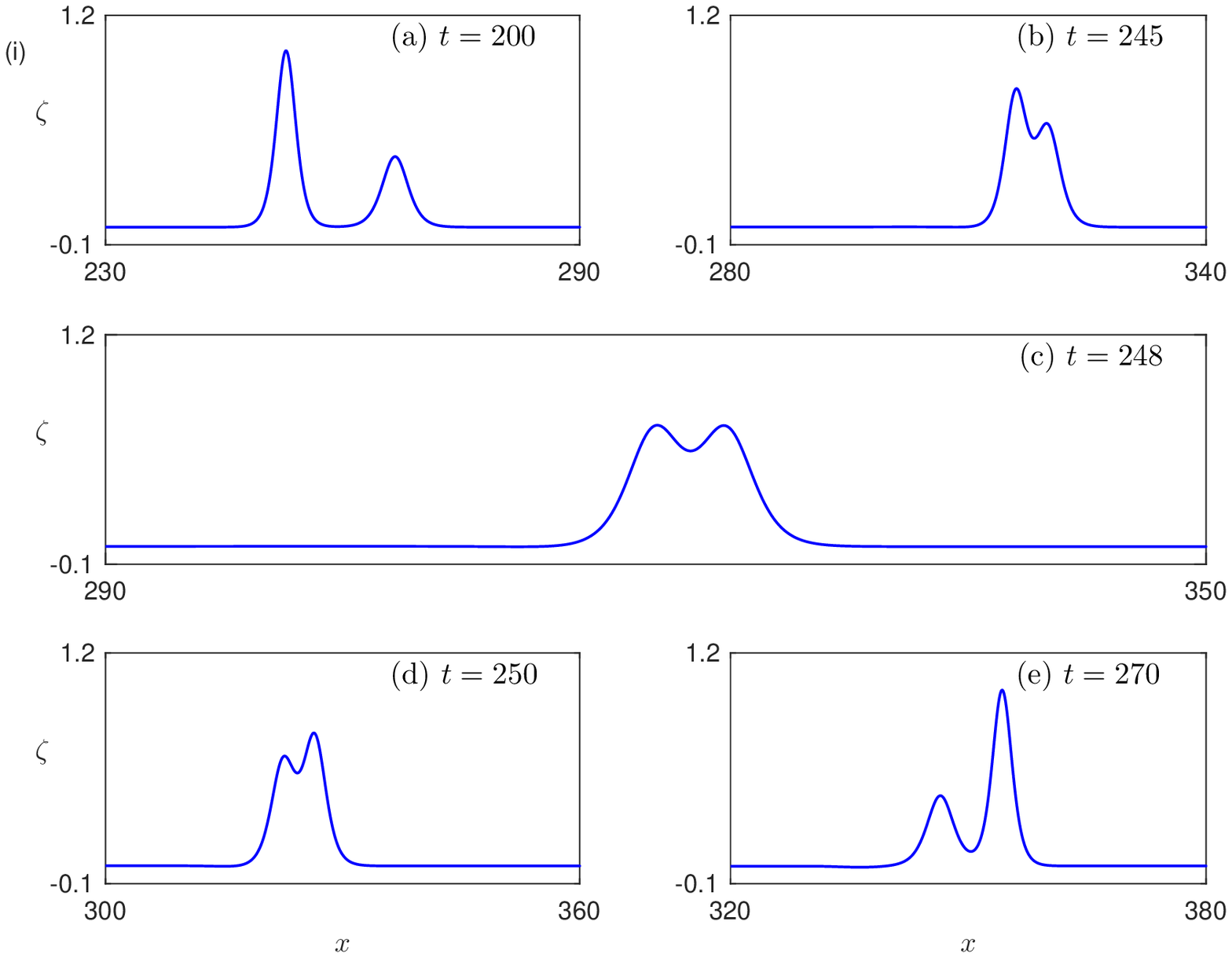} \\ \vspace{0.4cm}
  \includegraphics[scale=0.67]{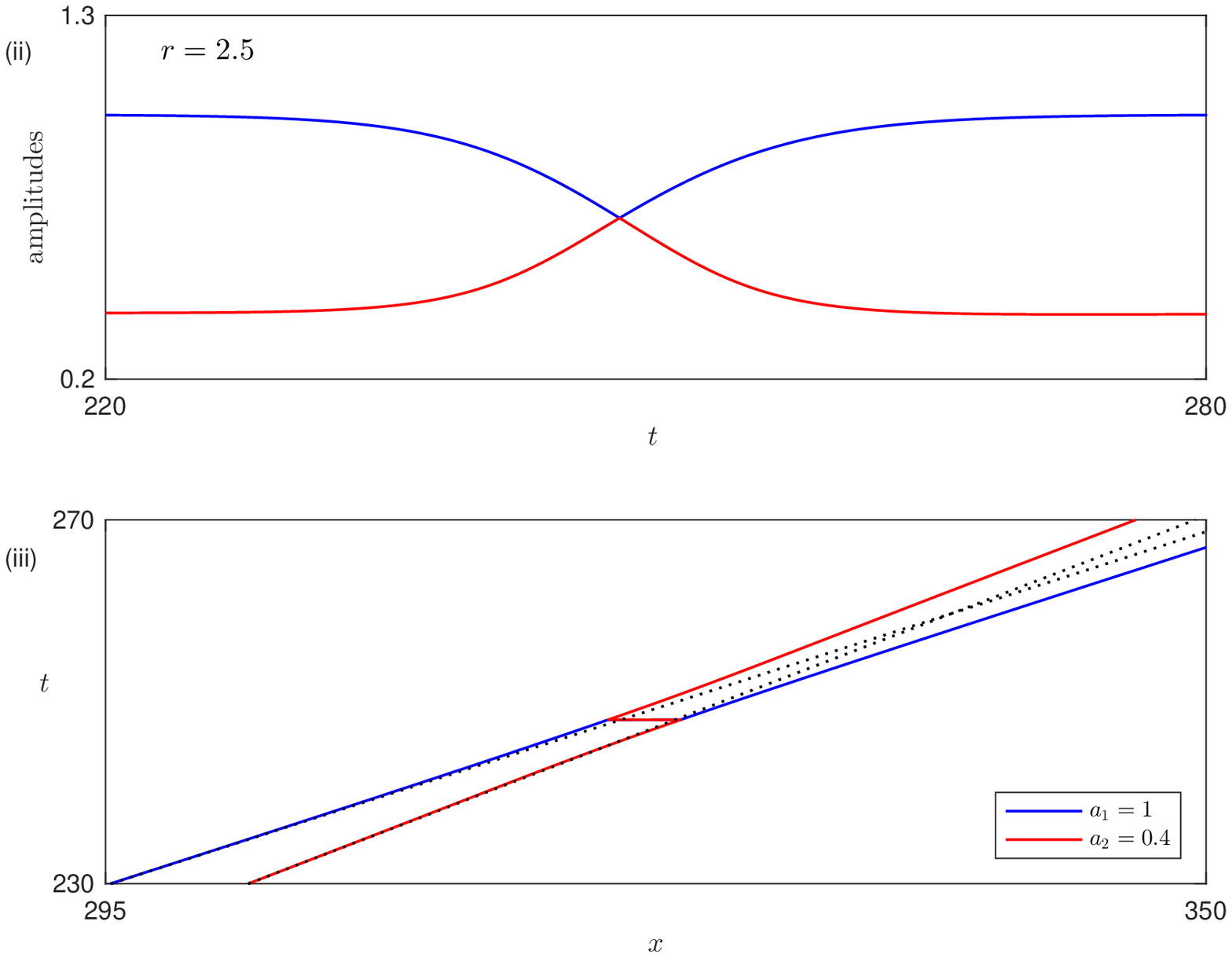}
  \caption{Overtaking collision, $r=2.5$ (Lax case (a)).
  (i): $\zeta$-profiles of the solution at various temporal instances, 
  (ii): Peak amplitudes of $\zeta$ as functions of $t$, (iii): Location of the peaks in the $x$,$t$-plane.}
  \label{fig34}
\end{figure}
place of the smaller there emerges a solitary wave closely resembling the larger one and vice versa.
During this interaction two distinct maxima of the solution exist at all times. Therefore, this type of 
interaction is analogous to the overtaking collision of two solitons of the KdV equation classified by Lax
in Lemma 2.3 of \cite{Lx} as case (a). In Figure \ref{fig34} we show an example of this type of collision
with $r=2.5$. The $\zeta$-profiles of the solution are depicted as functions of $x$ at several temporal
instances close to the interaction in Figure \ref{fig34}(i). In Figure \ref{fig34}(ii) we show the peak 
amplitudes of the two waves as functions of time. (In order to compute a peak amplitude with high accuracy
at each $t^{n}$, we locate a node $x_{i}$ where the fully discrete approximation $\zeta(x,t^{n})$ achieves
a local discrete maximum, solve $\zeta_{x}(x,t^{n})=0$ in  the vicinity of $x_{i}$ by Newton's method, and
find the peak value and its location $x^{*}$. For $r\leq 3.0967$ the code was able to find two distinct
maxima for all $t$.) After the interaction the larger solitary wave suffers a slight loss of amplitude.
while the smaller one gains a small amount. This is shown in the 
\begin{table}[ht]
\centering \footnotesize
\begin{tabular}{|c|c|c|c|c|c|} \hline  
Lax case & $r$   & $a_{1}$ & $a_{1}$ after & $a_{2}$ & $a_{2}$ after \\ \hline 
(a)      & 2.5   & 1       & 0.99966       & 0.4     & 0.40051       \\  
(b)      & 3.125 & 1       & 0.99965       & 0.32    & 0.32065        \\   
(c)      & 5.0   & 1       & 0.99976        & 0.2     & 0.20066       \\ \hline
\end{tabular}
\caption{Amplitudes of the solitary waves before and after the interaction 
($a_{1}=$ amplitude of the large s.w., $a_{2}=$ amplitude of the small s.w.).}  
\label{tbl35} 
\end{table} 
first line of Table \ref{tbl35}, where the values `$a_{1}$ after', `$a_{2}$ after' are the 
stabilized amplitudes 
of the large and small solitary wave, respectively, well after the interaction. \par 
Figure \ref{fig34}(iii) shows the paths of the locations of the peaks $x^{*}$ of the solitary 
waves in the $x,t$-plane. During the interaction the smaller (slower) solitary wave undergoes a
backward phase shift while the larger (faster) wave is displaced forward by a smaller amount. (The phase
shifts are measured relative to the positions of the peaks in the absence of interaction, indicated by the
dotted lines in Fig \ref{fig34}(iii).) The absolute values of the phase shifts ($\Delta x^{*}$) are recorded,
at the indicated temporal values, in Table \ref{tbl36}. 
\begin{table}[ht]
\centering \footnotesize
\begin{tabular}{|c|c|c|c|c|} \hline  
Lax case & $r$   & small s.w. & large s.w. & $t$      \\ \hline 
(a)      & 2.5   & 3.1        & 2.4        & 270      \\  
(b)      & 3.125 & 2.6        & 2.1        & 230      \\   
(c)      & 5.0   & 2.1        & 1.4        & 230      \\ \hline
\end{tabular}
\caption{Absolute values of phase shifts of the small and the large solitary waves measured at the 
indicated times.}  
\label{tbl36} 
\end{table}  
For values of $r\geq 3.9784$ (the values of $r$ where a transition in type occurs 
have been determined 
to four decimals) the interaction resembles the one shown in Figure \ref{fig35}, which
corresponds to $r=5$. The large solitary wave covers the small one completely, for a while there is only
one pulse visible, and finally the small solitary wave is re-emitted from the backside of the large one. 
Hence this type of interaction is analogous to that of case (c) of Lax's Lemma 2.3 for the KdV. As may
be seen in Figures \ref{fig35}(ii) and (iii) the code was able to detect two distinct local maxima initially
for $t\leq 172.50$ and again for $t\geq 192.47$. The stabilized amplitudes after the interaction and the
phase shifts at $t=230$ are given in Tables \ref{tbl35} and \ref{tbl36}, respectively. They resemble 
qualitatively those of case (a). 
\begin{figure}[t!]
  \centering
  \vspace{0.4cm}
  \includegraphics[scale=0.67]{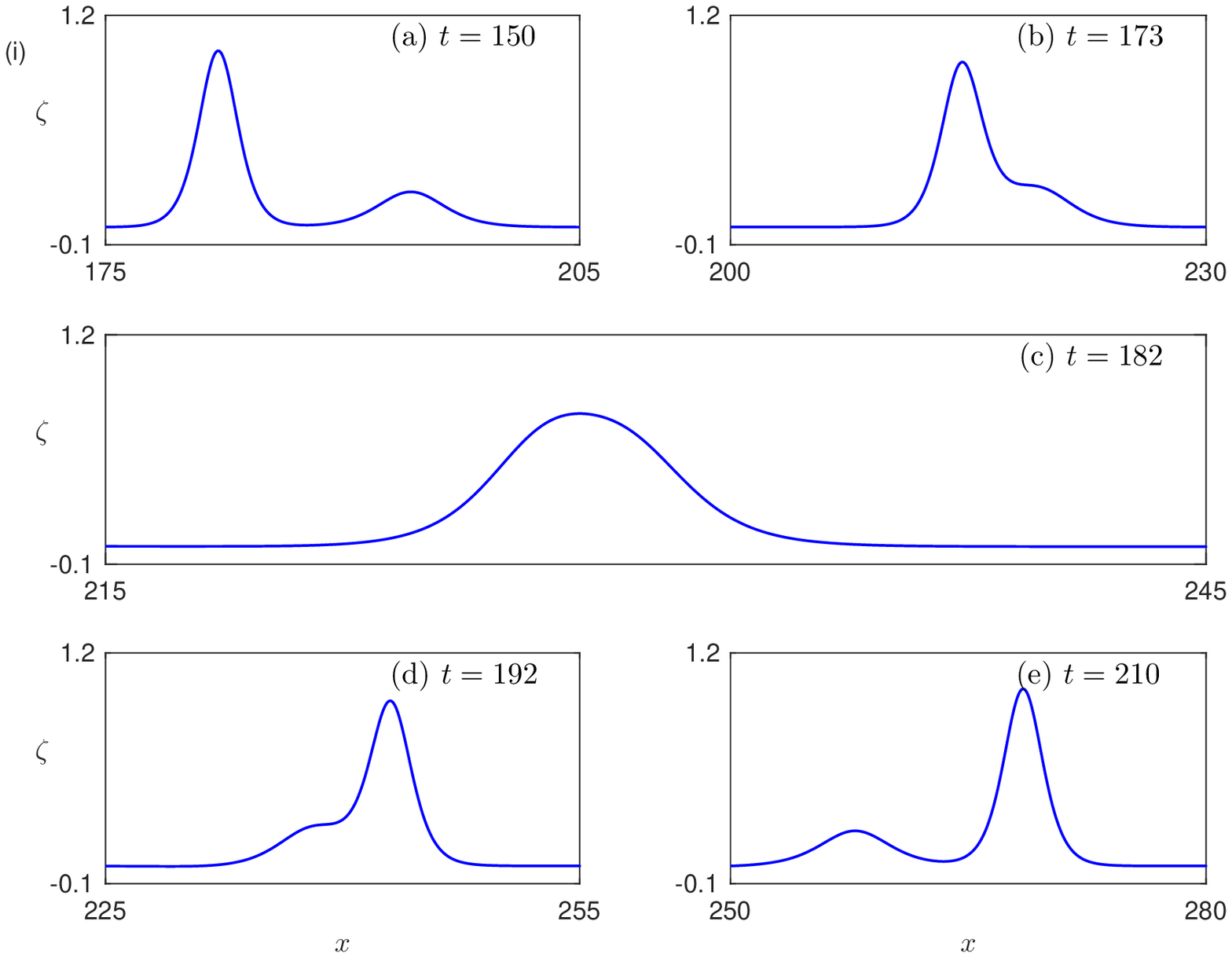}\\ \vspace{0.4cm}
  \includegraphics[scale=0.67]{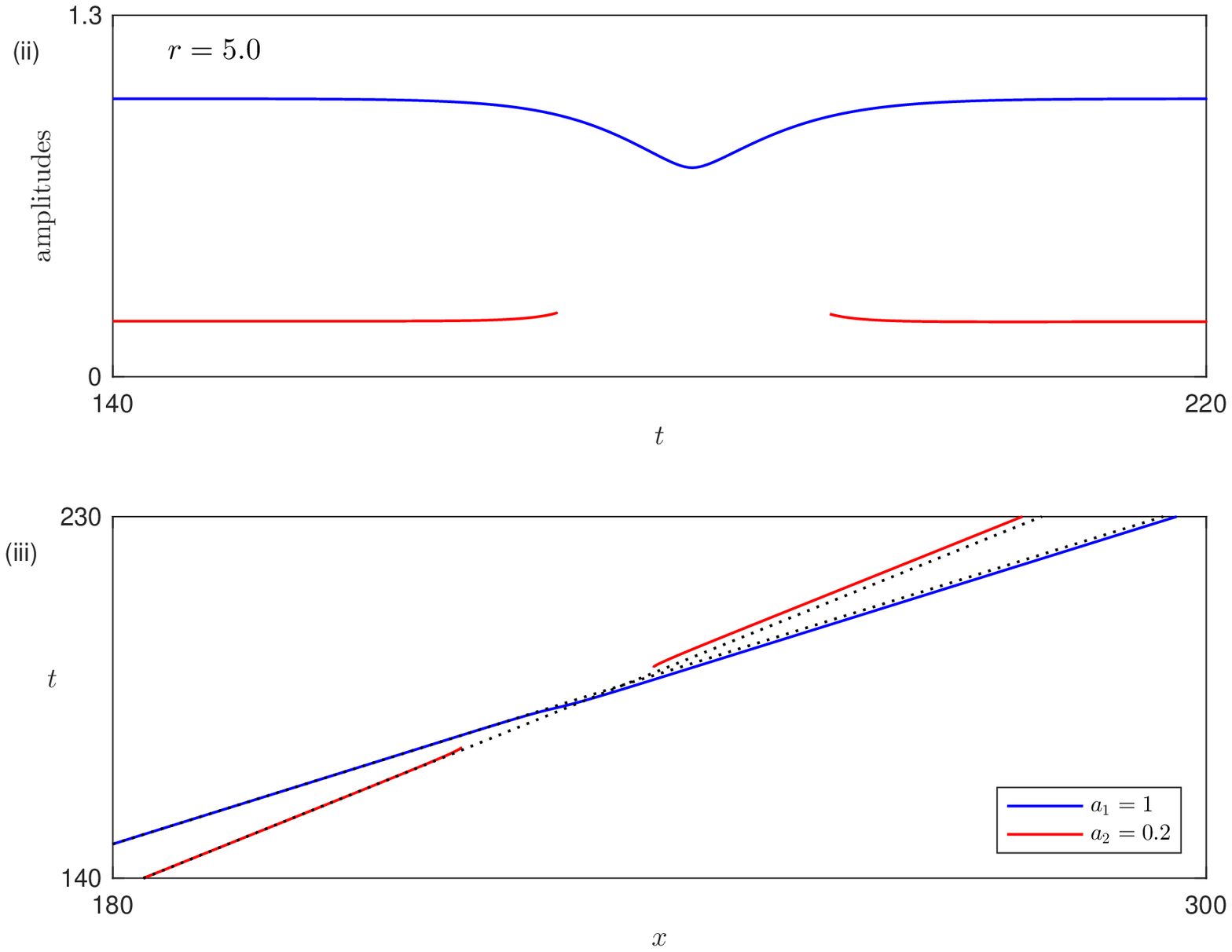}
  \caption{Overtaking collision, $r=5$ (Lax case (c)). (i): $\zeta$-profiles of the solution at various
  temporal instances, (ii): Peak amplitudes of $\zeta$ as functions of $t$, (iii): Location of peaks in the
  $x$,$t$-plane.}
  \label{fig35}
\end{figure}
\par 
In the range $3.1088\leq r\leq 3.9783$, the interaction looks like the one shown in Figure \ref{fig36},
which corresponds to the value $r=3.125$. 
 \begin{figure}[h!]
  \centering
  \vspace{0.4cm}
  \includegraphics[scale=0.67]{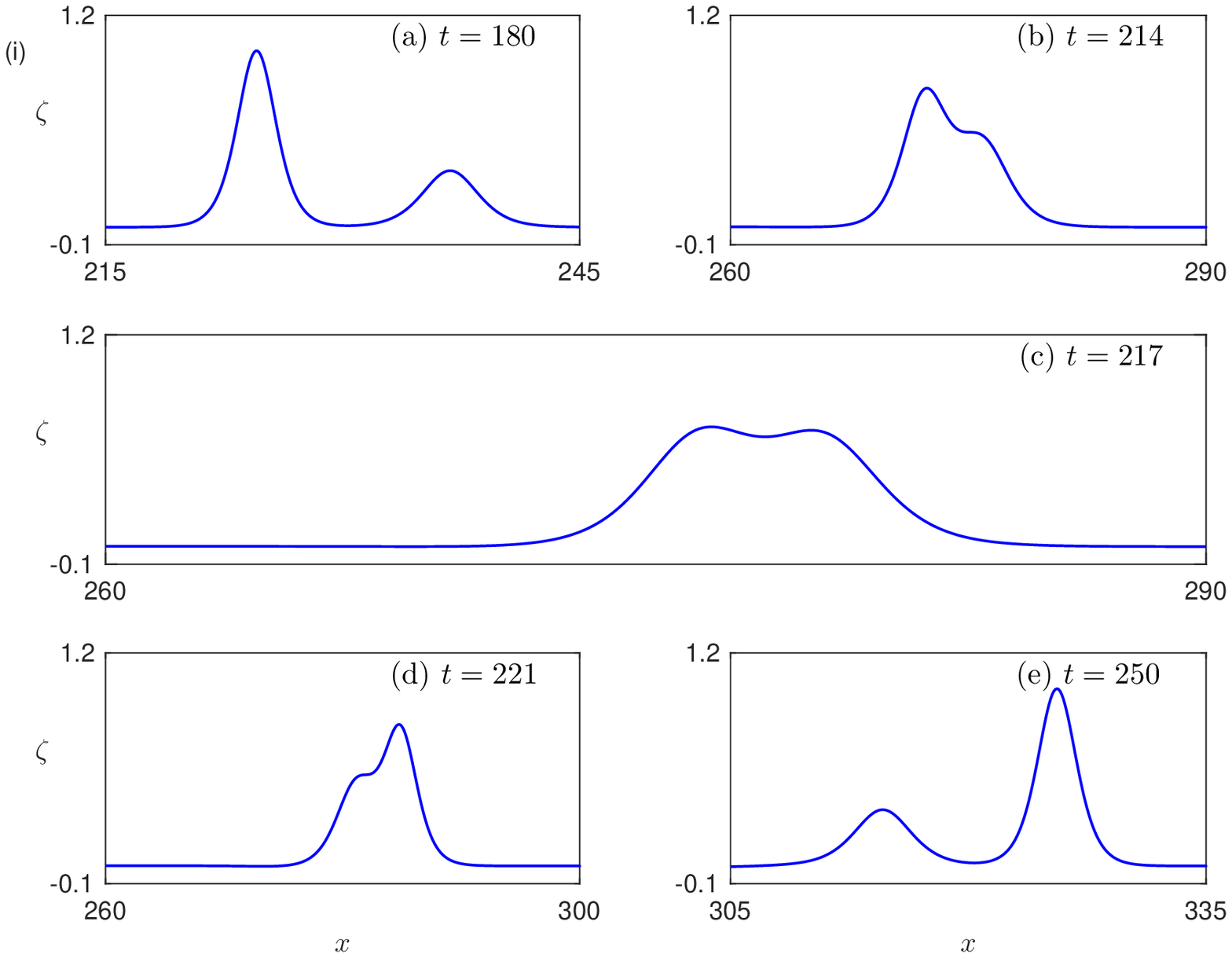}\\ \vspace{0.4cm}
  \includegraphics[scale=0.67]{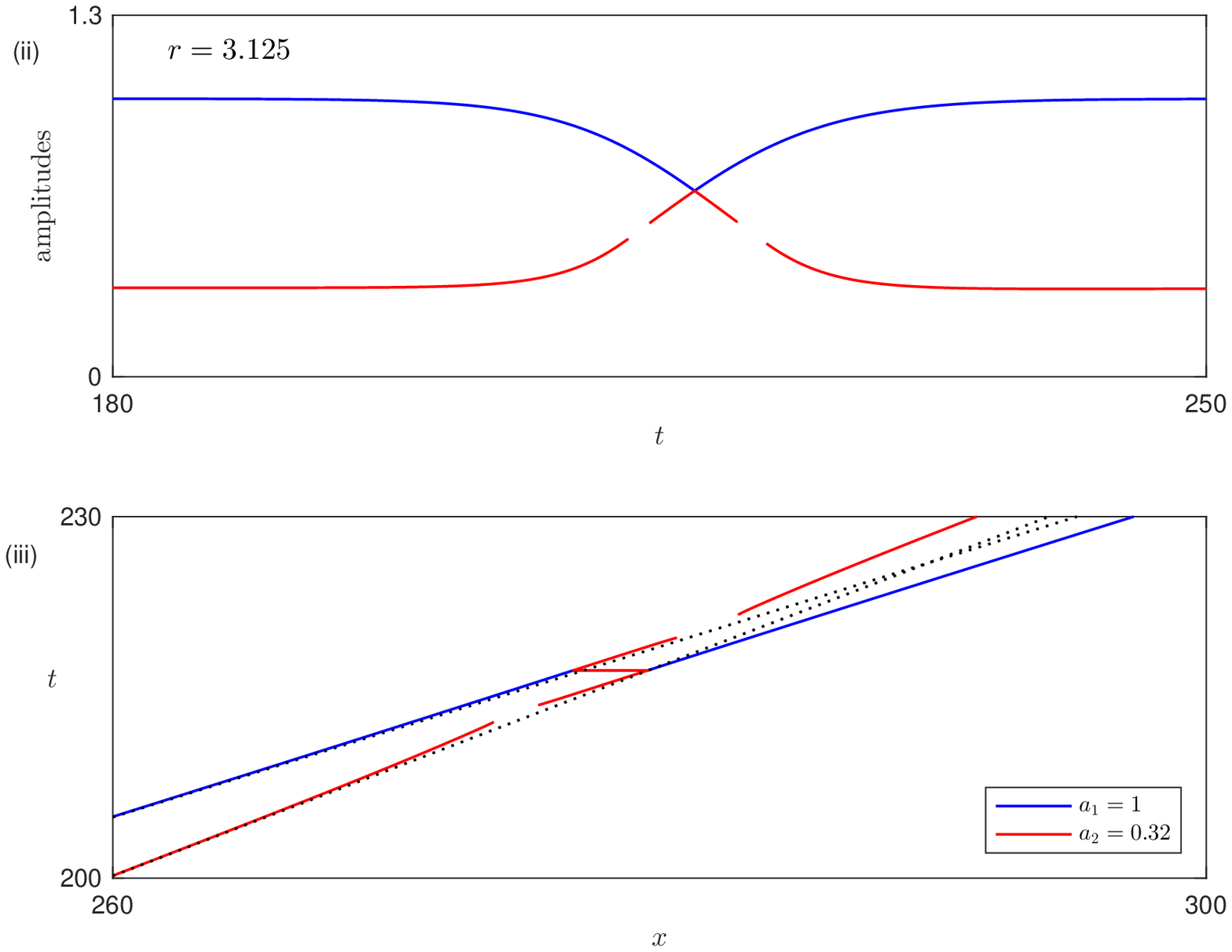}
  \caption{Overtaking collision, $r=3.125$ (Lax case (b)). (i): $\zeta$-profiles of the solution at
  various temporal instances, (ii): Peak amplitudes of $\zeta$ as function of $t$, 
  (iii): Location of the peaks in the $x$,$t$-plane.}
  \label{fig36}
\end{figure}Initially, the large solitary wave covers the small one and for 
a small temporal interval only one local maximum is observed. Subsequently, the smaller pulse reappears
and grows in amplitude, while the amplitude of the larger one diminishes. 
(During this phase two distinct local maxima exist.)
 After the two amplitudes take momentarily equal values, the process is repeated 
in reverse: 
The two waves exchange positions, the small one is absorbed for a small 
temporal interval by the
large one and is finally re-emitted as a separate small solitary wave from the backside 
of the larger one. Therefore this type of interaction, 
intermediate between the two previous types, is analogous to that of case
(b) in Lax's lemma, 
valid for the KdV. The event may be observed more clearly in the evolution of
the peak amplitudes shown in Fig. \ref{fig36}(ii). The code was able to detect two distinct local maxima 
up to $t=212.96$ and only one maximum for $t\in [212.97,214.35]$ when the smaller solitary wave is
absorbed by the pursuing larger wave. Subsequently, two distinct peaks are recorded again until
$t=219.95$ when the smaller wave is absorbed again by the larger one. The code was not able to detect
a second peak for $t\in[219.96,221.84]$. After this temporal interval two local maxima reappeared. The
quantitative scattering data for this interaction (amplitudes and phase shifts) are given in the middle
lines of Tables \ref{tbl35} and \ref{tbl36}, respectively; they resemble qualitatively the corresponding 
values of cases (a) and (c). \par 
It should be noted that whereas the signs of the phase shifts of the emerging solitary waves (the larger
wave is pushed forward while the smaller one is delayed) are the same as the ones observed in the case 
of the Euler equations in \cite{CGHHS} (and also in the case of the CB, \cite{AD1}), Table \ref{tbl35} 
shows that in all cases after the interaction the larger wave diminishes slightly in amplitude while
 \begin{figure}[b!]
  \centering
  \vspace{0.2cm}
  \includegraphics[scale=0.65]{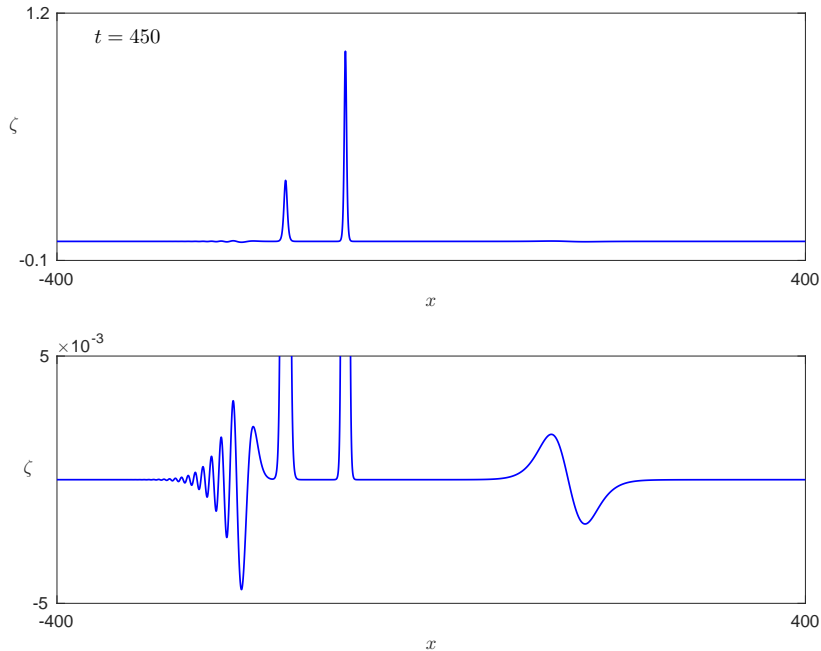}
  \caption{The dispersive tail and the wavelet in the case $r=3.125$.}
  \label{fig37}
\end{figure} 
similar to the Serre case.) \par 
As in the case of the Boussinesq systems and of the Euler equations, the interaction is 
{\em inelastic}. 
In all cases we observed that after the interaction two kinds of dispersive small-amplitude residuals were generated: A dispersive tail of small-wavelength, decaying in amplitude oscillations that travel to
the right trailing the 
solitary waves, and a single $N$-shaped, small-amplitude wavelet of large
wavelength that 
travels to the left. These are illustrated in Figure \ref{fig37} that shows the $\zeta$-profile of the solution 
in the case $r=3.125$ at $t=450$ and its magnification in the direction of the $\zeta$-axis.
(The solution has been translated periodically so that the solitary waves and the residuals appear near
the center of the figure.) The dispersive tail and the wavelet are of $O(10^{-3})$  in amplitude. This
structure of the residual was qualitatively the same for all values of $r$ that we tried in the 
various Lax cases. It should be noted that it resembles the residual of overtaking collisions observed
in the case of various Boussinesq systems, \cite{ADM1}, \cite{AD1}, but perhaps not the residual appearing
in numerical simulations of similar interactions in the case of the Euler equations; cf. e.g. Figure 19 
of \cite{CGHHS}, where only the wavelet seems to have been produced. \par 
Let us also point out that the intervals of $r$ in which the overtaking collisions resemble those of the
Lax cases (a), (b), and (c) depend on the value of, say, $a_{1}$ as well, i.e. the type of interaction
does not depend solely on $r$ as in the case of KdV. For example, when we took $a_{1}=1.5$, we found
that the interactions were of type (a) for $r=a_{1}/a_{2}$ less than $3.372$, of type (c) for $r>4.545$,
and of type (b) for values of $r$ in between. This is akin to what was observed for the Euler equations in
\cite{CGHHS}.       
As one final note of interest, one may observe that the temporal intervals 
[212.97,214.35] and  $[219.96,221.84]$, \normalsize
during  which there is apparently only one peak in the interaction of type (b) shown in 
\begin{figure}[ht!]
  \centering
  \vspace{0.4cm}
  \includegraphics[scale=0.67]{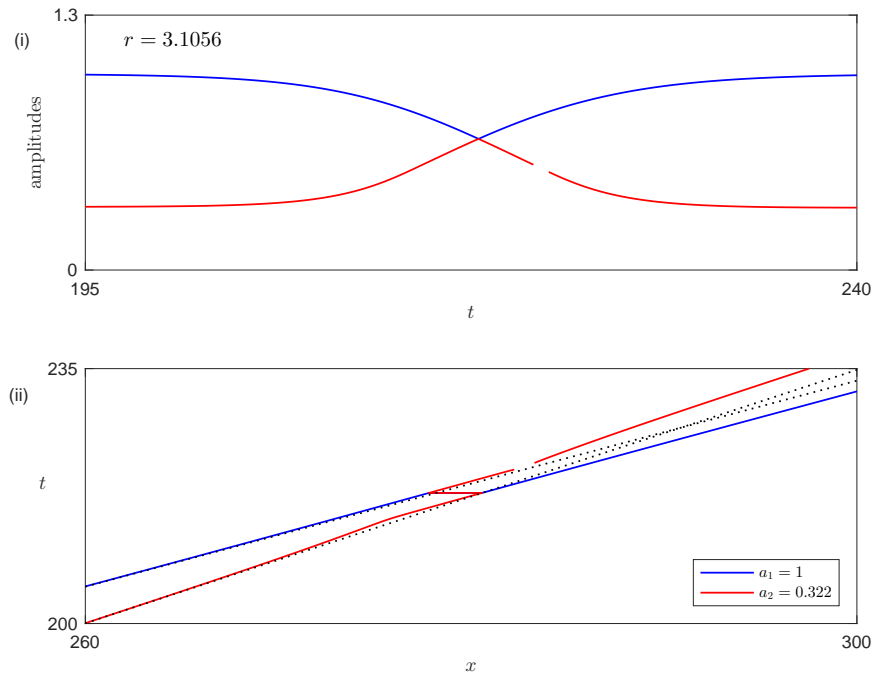}
  \caption{(i): Peak amplitudes of $\zeta$ as functions of $t$, (ii): Locations of the peaks in the 
  $x$,$t$-plane. Intermediate case between (a) and (b), $r=3.1056$.}
  \label{fig38}
\end{figure} 
Figure \ref{fig36}(ii) in the case $a_{1}=1$, $r=3.125$, are of unequal duration, 
the first one being smaller than the second. 
This led us to investigate whether these could  be values of $r$ for which the
first interval disappears but the second does not. We found that for 
$3.0968\leq r\leq 3.1087$ this was indeed the case, illustrated for $r=3.1056$ in 
the first one being smaller than the second. 
This led us to investigate whether these could  be values of $r$ for which the
first interval disappears but the second does not. We found that for 
$3.0968\leq r\leq 3.1087$ this was indeed the case, illustrated for $r=3.1056$ in 
Figure \ref{fig38}. In this type of interaction initially there are apparently two 
distinct local maxima that exchange heights and then the smaller peak is absorbed by
the larger one to reemerge later at the back of the large wave. Thus this interaction 
appears to be of an intermediate (transitional) type between cases (a) and (b) and we 
label it accordingly as case (ab).
\par 
In conclusion, Table \ref{tbl37} shows the intervals of $r$ (for $a_{1}=1$) 
and the corresponding type of 
interaction that we observed.
In the case (a) our code was always able to find two distinct local 
\begin{table}[ht] 
\centering \footnotesize
\begin{tabular}{|c|c|c|c|c|} \hline  
 & & transitional &          &                 \\ 
Cases & (a)  & (ab) &  (b) & (c)              \\ \hline 
\lower.2ex\hbox{$r$} & \lower.2ex\hbox{$(1,3.0967]$} & \lower.2ex\hbox{$[3.0968,3.1087]$} &  
\lower.2ex\hbox{$[3.1088,3.9783]$} & \lower.2ex\hbox{$[3.9784,+\infty)$}  \\ \hline 
\end{tabular}
\caption{Overtaking collisions, Serre equations, $a_{1}=1$, $r=a_{1}/a_{2}$. Lax cases and corresponding
intervals of $r$.} 
\label{tbl37} 
\end{table} 
maxima. In the transitional case (ab) there was only one temporal interval in which a unique local 
maximum was found, while in case (b) two such disjoint intervals were detected. These two intervals merge 
into a single larger interval in case (c).

\end{document}